\documentclass[12pt]{article}

\usepackage[table]{xcolor}
\usepackage{amsmath, amsthm}
\usepackage{amsfonts,amssymb}
\usepackage{hyperref}
\usepackage{authblk}
\usepackage[all]{xy}
\usepackage{enumitem}
\usepackage{stmaryrd}

\usepackage{tikz}
\usetikzlibrary{intersections}

\newcommand{\N}{\mathbb{N}}
\newcommand{\Z}{\mathbb{Z}}
\newcommand{\Q}{\mathbb{Q}}
\newcommand{\F}{\mathbb{F}}
\renewcommand{\P}{\mathbb{P}}
\newcommand{\R}{\mathbb{R}}
\newcommand{\C}{\mathbb{C}}
\newcommand{\K}{\mathbb{K}}

\renewcommand{\l}{\mathfrak{l}}

\newcommand{\Gal}{\operatorname{Gal}}
\newcommand{\GQ}{\Gal(\overline \Q / \Q)}
\newcommand{\GL}{\operatorname{GL}}
\newcommand{\SL}{\operatorname{SL}}
\newcommand{\PGL}{\operatorname{PGL}}
\newcommand{\PSL}{\operatorname{PSL}}
\newcommand{\GLFl}{\GL_2(\F_\ell)}
\newcommand{\SLFl}{\SL_2(\F_\ell)}
\newcommand{\PGLFl}{\PGL_2(\F_\ell)}

\newcommand{\Fl}{{\F_\ell}}
\newcommand{\Flx}{{\F_\ell^*}}
\newcommand{\Ker}{\operatorname{Ker}}
\renewcommand{\Im}{\operatorname{Im}}

\newcommand{\disc}{\operatorname{disc}}
\newcommand{\Res}{\operatorname{Res}}
\newcommand{\Aut}{\operatorname{Aut}}
\newcommand{\Sym}{\operatorname{Sym}}
\newcommand{\Hom}{\operatorname{Hom}}
\newcommand{\Ext}{\operatorname{Ext}}
\newcommand{\ab}{\operatorname{ab}}
\newcommand{\Tra}{\operatorname{Tra}}
\newcommand{\Inf}{\operatorname{Inf}}
\newcommand{\ord}{\operatorname{ord}}

\newcommand{\TableMinPolyVar}{x}

\def\saint{\mathaccent"7017 }

\newtheorem{thm}{Theorem}
\newtheorem{lem}[thm]{Lemma}
\newtheorem{pro}[thm]{Proposition}

\theoremstyle{definition}
\newtheorem{rk}[thm]{Remark}

\makeatletter
\newcommand{\subjclass}[2][2010]{%
  \let\@oldtitle\@title%
  \gdef\@title{\@oldtitle\footnotetext{#1 \emph{Mathematics subject classification:} #2}}%
}
\newcommand{\keywords}[1]{%
  \let\@@oldtitle\@title%
  \gdef\@title{\@@oldtitle\footnotetext{\emph{Key words and phrases.} #1.}}%
}
\let\c@table\c@equation
\let\c@figure\c@equation
\makeatother
\makeatother

\title{Certification of modular Galois representations}
\author{Nicolas Mascot\thanks{\href{mailto:n.a.v.mascot@warwick.ac.uk}{n.a.v.mascot@warwick.ac.uk}}}
\affil{\scriptsize{University of Warwick, Coventry CV4 7AL, UK. Formerly IMB, Universit\'e Bordeaux 1, UMR 5251, F-33400 Talence, France.
CNRS, IMB, UMR 5251, F-33400 Talence, France.
INRIA, project LFANT, F-33400 Talence, France.}}

\begin{document}

\maketitle

\begin{abstract}
We show how the output of the algorithm to compute modular Galois representations described in \cite{Moi} can be certified. We have used this process to compute certified tables of such Galois representations obtained thanks to an improved version of this algorithm, including representations modulo primes up to $31$ and representations attached to a newform with non-rational (but of course algebraic) coefficients, which had never been done before. These computations take place in the Jacobian of modular curves of genus up to $26$. The resulting data are available on the author's \href{http://www2.warwick.ac.uk/fac/sci/maths/people/staff/mascot/galreps}{webpage}.
\end{abstract}

\renewcommand{\abstractname}{Acknowledgements}
\begin{abstract}
The computations presented here would not have been amenable without Bill Allombert, who suggested to me the idea of step-by-step polynomial reduction, and Karim Belabas and Denis Simon, who provided me their \cite{idealsqrt} script. I also thank J. Kl\"uners for the useful discussions I have had with him in July 2015 in Oberwolfach about the algorithmic computation of Galois groups, D. Holt for his help in permutation group theory, and my friend and colleague A. Page for the clever suggestions that he provided me. Finally, I wish to address my thanks to the anonymous reviewer of the previous version of this article for suggesting a much simpler method to certify some Galois groups (cf. section \ref{Galois_geom}) and more generally for his insightful suggestions, which have helped me to make this article clearer.

This research was supported by the French ANR-12-BS01-0010-01 through the project PEACE, by the DGA ma{\^\i}trise de l'information, and by the EPSRC Programme Grant EP/K034383/1 ``LMF: L-Functions and Modular Forms''. 

The computations presented in this paper were partly carried out using the PlaFRIM experimental testbed, being developed under the Inria PlaFRIM development action with support from LABRI and IMB and other entities: Conseil R\'egional d'Aquitaine, FeDER, Universit\'e de Bordeaux and CNRS (see \href{https://plafrim.bordeaux.inria.fr/}{https://plafrim.bordeaux.inria.fr/}), and partly on the Warwick mathematics institute computer cluster provided by the aforementioned EPSRC grant. The computer algebra packages used were \cite{Sage}, \cite{gp} and \cite{Magma}.
\end{abstract}

\newpage

We begin with a short summary about Galois representations attached to modular forms and how we used these in \cite{Moi} to compute Fourier coefficients of modular forms in section \ref{Background}. This computation becomes much easier if the polynomial in $\Q[x]$ defining the representation and computed by the algorithm along the way is reduced, and we show new ideas to do so efficiently in section \ref{Polred}. We then show in section \ref{Sect_certif} how the outputs of this computation can be formally certified. Finally, we comment on the use of this certification method on our own data in the last section \ref{Tables}.

\section{Introduction}\label{Background}

Let $f = q+\sum_{n=2}^{+\infty} a_n q^n \in S_k\big(\Gamma_1(N),\varepsilon\big)$ be a classical newform of weight $k \in \N_{\geqslant 2}$, level $N \in \N_{\geqslant 1}$ and nebentypus $\varepsilon$. Jean-Pierre Serre conjectured and Pierre Deligne proved in \cite{Del} that for every finite prime $\l$ of the number field $K_f = \Q(a_n, \ n \geqslant 2)$ spanned by the coefficients $a_n$ of the $q$-expansion of $f$ at infinity, there exists a continuous Galois representation
\[ \GQ \longrightarrow \GL_2(\Z_{K_{f,\l}}) \]
which is unramified outside $\ell N$ and such that the image of any Frobenius element at $p \nmid \ell N$ has characteristic polynomial $x^2-a_p x + \varepsilon(p) p^{k-1} \in \Z_{K_{f,\l}}[x]$, where $\Z_{K_{f,\l}}$ denotes the $\l$-adic completion of the ring of integers $\Z_{K_f}$ of $K_f$, and $\ell$ is the rational prime lying below $\l$.

Let $\F_\l$ be the residue field of $\l$. By reducing the above $\l$-adic Galois representation modulo $\l$, we get a modulo $\l$ Galois representation
\[ \rho_{f,\l} \colon \GQ \longrightarrow \GL_2(\F_\l), \]
which is unramified outside $\ell N$ and such that the image of any Frobenius element at $p \nmid \ell N$ has characteristic polynomial $x^2-a_p x + \varepsilon(p) p^{k-1} \in \F_\l[x]$. In particular, the trace of this image is $a_p \bmod \l$.

In \cite{Moi}, we described an algorithm based on ideas from the book \cite{CE} edited by Jean-Marc Couveignes and Bas Edixhoven to compute such modulo $\l$ Galois representations, provided that the image of the Galois representation contains $\SL_2(\F_\l)$ and that $k < \ell$. This gives a way to quickly compute the coefficients $a_p$ modulo $\l$ for huge primes $p$. We have used this algorithm to compute representations attached to forms of level 1 for $\ell$ up to $31$. 

The condition that the image of the Galois representation contain $\SL_2(\F_\l)$ is generically satisfied. Indeed, by \cite[theorem 2.1]{RibSL2} and \cite[lemma 2]{SwD}, for any non-CM newform $f$ (and in particular for any newform $f$ of level $1$), the image of the representation $\rho_{f,\l}$ contains $\SL_2(\F_\l)$ for almost every $\l$. The finitely many $\l$ for which $\SL_2(\F_\l) \not \subset \Im \rho_{f,\l}$ are called \emph{exceptional primes} for $f$, and we exclude them. They were explicitly determined by Sir Peter Swinnerton-Dyer in \cite{SwD} for the known\footnote{According to Maeda's conjecture (cf \cite{Maeda2000}), there are only 6 such forms, namely $\Delta$, $E_4 \Delta$, $E_6 \Delta$, $E_8 \Delta$, $E_{10} \Delta$ and $E_{14} \Delta$, of respective weights $12$, $16$, $18$, $20$, $22$ and $26$.} newforms $f$ of level $1$ whose coefficients $a_n$ are rational. In our case, this means we exclude $\l = 23$ for $f=\Delta$ and $\l=31$ for $f = E_4 \Delta$.

\pagebreak

In what follows,we will assume that the inertial degree of $\l$ is $1$, so that $\F_\l = \Fl$. Indeed, although there is no theoretical obstacle to allowing primes of higher degree, we will have to deal explicitly with objects such as polynomials whose roots are indexed by $\Fl^2\setminus\{(0,0)\}$ and whose Galois group is $\GLFl$, and this already requires considerable work when $\F_\l = \Fl$.

Our algorithm relies on the fact that if $k < \ell$, then the Galois representation $\rho_{f,\l}$ is afforded with multiplicity $1$ by a subspace $V_{f,\l}$ of the $\ell$-torsion of the Jacobian $J_1(\ell N)$ of the modular curve $X_1(\ell N)$ under the natural $\GQ$-action, cf. \cite[proposition 9.3.2]{Gross} and \cite[section 1]{Moi}.

The algorithm first computes the number field $L = \overline \Q^{\Ker \rho_{f,\l}}$ cut out by the Galois representation, by evaluating a well-chosen function $\alpha \in \Q\big( J_1(\ell N) \big)$ in the nonzero points of $V_{f,\l}$ and forming the polynomial
\[ F(x) =\prod_{\substack{v \in V_{f,\l} \\ v \neq 0}} \big(x-\alpha(v) \big) \in \Q[x] \]
of degree $\ell^2-1$ whose decomposition field is $L$. The algorithm then uses a method from T. and V. Dokchitser (cf \cite{Dok}) to compute the image of the Frobenius element at $p$ given a rational prime $p \nmid \ell N$. This method involves the computation of a family of resolvents
\[ \Gamma_C(x) = \prod_{\sigma \in C} \left( x - \sum_{\substack{v \in V_{f,\l} \\ v \neq 0}} h\big(\alpha(v) \big) \, \alpha(\sigma \cdot v) \right) \in \Q[x] \]
indexed by the conjugacy classes $C$ of $\GLFl$, where $h(x) \in \Z[x]$ is some fixed polynomial. These resolvents, which we will refer to as the Dokchitsers' resolvents, can then be used to determine which class the Frobenius element at $p$ lies in for almost all $p \in \N$.

\begin{rk}
Actually, in order to obtain certified results, we will see that we should certify the polynomial $F(x)$ in the sense of section \ref{Sect_certif} before computing the Dokchitsers' resolvents.
\end{rk}

Unfortunately, the output of the algorithm, although correct beyond reasonable doubt (cf. \cite{Moi}, end of section 1), is not certified since it relies on the identification of floating point numbers as rational numbers. The purpose of this article is to show how these computations can be formally certified subsequently. As a side effect, we also obtain much tidier outputs. 

\paragraph*{A word on notation}\mbox{}

All along this article, we will be dealing with two versions of most of the objects in play, namely the actual value of this object, and the version computed by the algorithm described above. For instance, the function $\alpha \in \Q\big( J_1(\ell N) \big)$ being fixed, the polynomial
\[ F(x) =\prod_{\substack{v \in V_{f,\l} \\ v \neq 0}} \big(x-\alpha(v) \big) \in \Q[x] \]
is a well-defined object attached to $\alpha$, $f$ and $\l$, but what the algorithm outputs is an approximate version of this polynomial over $\C$, whose coefficients are then non-rigorously identified as rational numbers. Following the reviewer's comments on an older version of this article, we will denote the ``true'' value of $F(x)$ with an aureole, $\saint F(x)$, so as to stress its ``heavenly unattainable nature'' (as the reviewer put it), and we will reserve the notation $F(x)$ to the polynomial ``guessed'' by the algorithm, and similarly for the other objects at play. We will follow this convention from now on, and we hope that doing so will reduce the confusion between the two versions of each object, and make our certification process clearer.

\section{Reducing the polynomials}\label{Polred}

Unfortunately, the coefficients of the polynomial $F(x)$ produced by the algorithm described in \cite{Moi} tend to have larger and larger height as $\ell$ grows. More precisely, in practice this polynomial is of the form
\[ F(x) = x^{\deg F} + \frac1d \sum_{i<\deg F} c_i x^i, \]
where $d$ is an (unfortunately large) positive integer and the $c_i$ are integers whose gcd with $d$ is several orders of magnitude smaller than $d$; in other words, apart from the leading one, these coefficients roughly all have the same denominator, with a few ``accidental'' simplifications here and there. The following table, which shows the genus $g = \frac{(\ell-5)(\ell-7)}{24}$ of the modular curves $X_1(\ell)$ and the rough number $h \approx \log_{10} d$ of decimal digits in the denominator $d$ of the polynomials $F(x)$ associated to newforms of level $N=1$ that we computed using the algorithm described in \cite{Moi}, seems to indicate the heuristic $h \approx g^{2.5}$:

\begin{center}
\begin{tabular}{|c|c|c|}
\hline
$\ell$ & $g$ & $h$ \\
\hline
\phantom{\quad} 11 \phantom{\quad} & \phantom{\quad} 1 \phantom{\quad} & \phantom{\quad} 0 \phantom{\quad} \\
13 & 2 & 5 \\
17 & 5 & 50 \\
19 & 7 & 150 \\
23 & 12 & 500 \\
29 & 22 & 1800 \\
31 & 26 & 2500 \\
\hline
\end{tabular}
\end{center}

\bigskip

While this is rather harmless for $\ell \leqslant 17$, it makes the Dokchitser's method intractable as soon as $\ell \geqslant 29$. It is thus necessary to reduce this polynomial, that is to say to compute another polynomial whose splitting field is isomorphic to the splitting field of $F(x)$ but whose coefficients are much nicer. An algorithm to perform this task based on LLL lattice reduction is described in \cite[section 4.4.2]{GTM138} and implemented in \cite{gp} under the name \verb?polred?. Its complexity is polynomial in the degree and the height of the coefficients, provided that the factorisation of the discriminant of the corresponding field is know, which is the case for us. However, the polynomial $F(x)$ has degree $\ell^2-1$ and tends to have really large coefficients, and this makes \verb?polred? choke on it, even for small values of $\ell$. Indeed, the fact that \verb?polred? is based on LLL reduction means that its execution time is especially sensitive to the degree of the polynomial.

On the other hand, it would be amenable to apply the \verb?polred? algorithm to the polynomial
\[ \saint F^{\text{proj}}(x) = \prod_{W \in \P(V_{f,\l})} \left(x- \sum_{\substack{w \in W \\ w \neq 0}} \alpha(w) \right) \in \Q[x] \]
whose splitting field is\footnote{To be precise, it is clear that the splitting field of $\saint F^{\text{proj}}(x)$ is contained in the number field $\saint L^{\text{proj}}$ cut out by the projective representation. Very often, this containment is an equality and so $\saint F^{\text{proj}}(x)$ is irreducible, but it may sometimes happen that this containment is proper, in which case $\saint F^{\text{proj}}(x)$ becomes reducible over $\Q$. We can work around this pathological behaviour by replacing the summation over $W$ in the definition of $\saint F^{\text{proj}}(x)$ by another symmetric combination (e.g. a product), or by applying a Tschirnhausen transform. For notational convenience, we will henceforth assume that no such problem is encountered; should this not be the case, the necessary modifications are completely straightforward.} the number field $\saint L^\text{proj}$ cut out by the \emph{projective} Galois representation
\[ \saint \rho^\text{proj}_{f,\l} \colon \xymatrix{ \GQ \ar[r]^{\saint \rho_{f,\l}} & \GL_2(\Fl) \ar@{->>}[r] & \PGL_2(\Fl) } \]
since the degree of this polynomial is only $\ell+1$. Unfortunately, this projective version of the representation does not contain enough information to recover\footnote{One could at most recover these values with a sign ambiguity, as in \cite{CE}.} the values of $a_p \bmod \l$.

However, we noted in \cite[section 3.7.2]{Moi} that if $S \subset \Flx$ denotes the largest subgroup of $\F_\ell^*$ such that $S \not \ni -1$, then the knowledge of the quotient representation 
\[ \saint \rho^S_{f,\l} \colon \xymatrix{ \GQ \ar[r]^{\saint \rho_{f,\l}} & \GL_2(\Fl) \ar@{->>}[r] & \GL_2(\Fl)/S }, \]
combined with the fact that the image in $\GL_2(\Fl)$ of a Frobenius element at $p$ has determinant $p^{k-1} \varepsilon(p) \bmod \l$, is enough to recover $\saint \rho_{f,\l}$ and hence the values of $a_p$ mod~$\l$. It is therefore enough for our purpose to compute this quotient representation, first by forming the polynomial
\[ \saint F^S(x) = \prod_{\substack{Sv \in V_{f,\l}/S \\ v \neq 0}} \left(x- \sum_{s \in S} \alpha(sv) \right) \in \Q[x], \]
whose splitting field is the number field $\saint L^S$ cut out by $\saint \rho^S_{f,\l}$, and then by applying the Dokchitsers' method on it in order to compute the images of the Frobenius elements by $\saint \rho^S_{f,\l}$, cf. \cite[section 3.7.2]{Moi}.

Note that since we assumed that $f$ and $\l$ are such that $\saint \rho_{f,\l}$ is not exceptional\footnote{In the sense that its image contains $\SL_2(\Fl)$.}, the quotient representations $\saint \rho^{S}_{f,\l}$ is surjective. Indeed, since $f$ is a form of level $N=1$ and of even weight, the determinant of $\rho_{f,\l}$ is an odd power of the mod $\ell$ cyclotomic character. In particular, the polynomial $\saint F^S(x)$ is irreducible over $\Q$.

Also note that the complex roots of $\saint F(x)$ are approximately known as an output of the algorithm \cite{Moi}, and so is their indexation by $V_{f,\l} - \{0\}$. We thus have an indexation of the roots of $F(x)$ by $V_{f,\l}-\{0\}$, and so we can compute an approximation $F^S(x) \in \Q[x]$ of $\saint F^S(x)$ by grouping the roots, expanding over $\C$, and guessing the coefficients by continued fractions just like for $F(x)$.

In practice, the coefficients of $F^S(x)$ have roughly the same denominator as the ones of $F(x)$, so we are not improving anything on this side, but of course the degree of $F^S(x)$ can be much smaller, so we may try to \verb?polred? it. Let $\ell-1 = 2^r s$ with $s \in \N$ odd. Since we have $\vert S \vert = s$, the degree of $F^S$ is $2^r (\ell+1)$, so  \verb?polred?ing is amenable in the cases $\ell =19$ or $23$, but the cases $\ell = 29$ or $31$ remain impractical.

For these remaining cases, Bill Allombert suggested to the author that one can still reduce $F^S(x)$ in several steps, as we now explain. Since $\Flx$ is cyclic, we have a filtration
\[ \Flx = S_0 \underset{2}{\supset} S_1 \underset{2}{\supset} \cdots \underset{2}{\supset} S_r = S \]
with $[S_{i} : S_{i+1}] = 2$ for all $i$, namely
\[ S_i = \{ x^{2^i}, x \in \Flx \}. \]
For each $i \leqslant r$, let us define
\[ \saint F_i(x) = \prod_{\substack{S_i v \in V_{f,\l}/S_i \\ v \neq 0}} \left(x- \sum_{s \in S_i} \alpha(sv) \right) \in \Q[x], \]
let $F_i(x) \in \Q[x]$ be guesses for $\saint F_i(x)$ obtained as for $F^S(x)$ above, let
\[ \saint K_i = \Q[x] / \saint F_i(x), \qquad K_i = \Q[x] / F_i(x), \]
and let $\saint L_i$ (resp. $L_i)$ be the normal closure of $\saint K_i$ (resp. $K_i$), so that $\saint L_i$ the number field cut out by the quotient representation
\[ \saint \rho^{S_i}_{f,\l} \colon \xymatrix{ \GQ \ar[r]^{\saint \rho_{f,\l}} & \GL_2(\Fl) \ar@{->>}[r] & \GL_2(\Fl)/S_i }. \]
In particular, we have $\saint \rho_{f,\l}^{S_0} = \saint \rho_{f,\l}^\text{proj}$, $\saint L_0 = \saint L^\text{proj}$, and we are looking for a nice model of $K_r$.

Note that again because $f$ is of level $N=1$, and is not exceptional mod $\l$, the polynomials $\saint F_i(x)$ are irreducible over $\Q$, and so $\saint K_i$ is indeed a field. We assume that the $F_i(x)$ are also irreducible.

\bigskip

By construction, the degree of $\saint K_i$ over $\Q$ is $\# \big( (V_{f,\l}-\{0\})/S_i \big) = 2^i (\ell+1)$, so the fields $\saint K_i$ fit in an extension tower
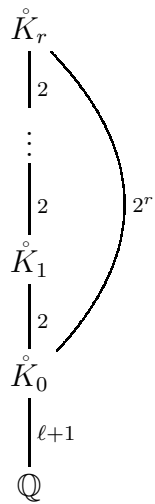
\[ \xymatrix{\saint K_r \ar@{-}[d]^2 \ar@{-}@/^3pc/[ddd]^{2^r} \\ \vdots \ar@{-}[d] ^2 \\ \saint K_1 \ar@{-}[d]^2 \\ \saint K_0 \ar@{-}[d]^{\ell+1} \\ \Q } \]
and we are going to \verb?polred? the polynomials $F_i(x)$ along this tower recursively from bottom up.

First, we apply directly the \verb?polred? algorithm to $F_0(x)=F^\text{proj}(x)$. Since the degree of this polynomial is only $\ell+1$, this is amenable, as mentioned above, and yields a monic reduced polynomial in $\Z[x]$.

Then, assuming we have managed to reduce $F_i(x)$, we have a nice model for $K_i = \Q[x] / F_i(x)$, and so we can factor $F_{i+1}(x)$ over $K_i$. Since the extension \linebreak $K_{i+1}=\Q[x]/F_{i+1}(x)$ should be quadratic over $K_i$, there must be at least one factor of degree $2$. Let $G_{i+1}(x)$ be one of those, and let $\Delta_i \in K_i$ be its discriminant, so that we have
\[ K_{i+1}\simeq K_i[x]/G_{i+1}(x) \simeq K_i \big(\sqrt{\Delta_i}\big). \]
In order to complete the recursion, all we have to do is to strip $\Delta_i$ from the largest square factor we can find, say $\Delta_i = A_i^2 \delta_i$ with $A_i, \delta_i \in K_i$ and $\delta_i$ as small as possible. Indeed we then have $K_{i+1}=K_i \big(\sqrt{\delta_i}\big)$, and actually even $K_{i+1}=\Q\big(\sqrt{\delta_i}\big)$ unless we are very unlucky\footnote{In practice, the case $K_{i+1} \supsetneq \Q\big(\sqrt{\delta_i}\big)$ has never happened to us. Should it happen, it can be corrected by multiplying $\delta_i$ by the square of an (hopefully small) element in $K_i$.}, so that if we denote by $\chi_i(x) \in \Q[x]$ the minimal polynomial of $\delta_i$, then we have
\[ K_{i+1} \simeq \Q[x] / \chi_i(x^2), \]
so that $\chi_i(x^2)$ is a reduced version of $F_{i+1}(x)$. If its degree and coefficients are not too big, we can even apply the \verb?polred? algorithm to this polynomial in order to further reduce it, which is what we do in practice.

In order to write $\Delta_i = A_i^2 \delta_i$, we would like to factor $\Delta_i$ in $K_i$, but even if $K_i$ is principal, this is not amenable whatsoever because $\Delta_i$ is huge. We can however consider the ideal generated by $\Delta_i$ in $K_i$, and remove its $\ell N$-part. The fractional ideal $\mathfrak{B}_i$ we obtain must then be a perfect square, since $K_{i+1}$ is unramified outside $\ell N$ (since $L$ is), and the very efficient \verb?idealsqrt? script from \cite{idealsqrt} can explicitly factor it into $\mathfrak{B}_i = \mathfrak{A}_i^2$. If $A_i$ denotes an element in $\mathfrak{A}_i$ close to being a generator of $\mathfrak{A}_i$ (an actual generator, if amenable, would be even better), then $\delta_i := \Delta_i / A_i^2$ is small.

\bigskip

We have thus managed to reduce our polynomials $F_i(x)$. In what follows, we will use the notation $F_i(x)$ to refer to the reduced versions, which are monic and lie in $\Z[x]$. They were each obtained from the non-reduced version by an explicit change of variable, and we can apply the same changes of variables to the ``true'' polynomials $\saint F_i(x)$, thus yielding new polynomials that we will denote by $\saint F_i(x)$ from now on.

\section{Certification of the computations}\label{Sect_certif}

The output of the algorithm relies on the identification as rational numbers of the coefficients of the polynomials $F_i(x)$ given in approximate form as floating-point numbers, by using continued fractions. In order to certify these results, it is thus necessary to make sure the we have correctly identified not only that the number fields cut out by the representation (i.e. that $K_i = \saint K_i$), but also the Galois action on the roots of the $F_i(x)$, else we would be doing nonsense with the Dokchitsers' resolvents $\Gamma_C(x)$.

For this, a first possibility consists in proving bounds on the height of the rational numbers that the algorithm will have to identify (e.g. the coefficients of $\saint F(x)$), and then to certify that the continued fraction identification process is correct, for instance by running the computation with high enough precision in $\C$ and controlling the round-off errors all along. Although it is indeed possible in theory to bound the height of these rational numbers by using Arakelov theory (cf. \cite[theorem 11.7.6]{CE}), this approach gives unrealistic titanic bounds and thus seems ominously tedious, especially as it requires controlling the round-off error in the linear algebra steps of K. Khuri-Makdisi's algorithms to compute in the modular Jacobian (cf. \cite[section 3.3]{Moi}). We have therefore not attempted to follow it.

Instead, we deemed it much better to first run the computations in order to obtain unproven results, and to prove these results afterwards. We explain in this section how to do so.

\subsection{Sanity checks}

Before attempting to prove the results, it is comforting to perform a few easy checks so as to ensure that they seem correct beyond reasonable doubt (cf. the end of section 1 in \cite{Moi}). Namely,

\begin{itemize}

\item Since we are working with a form of level $N=1$, the number field $\saint L$ cut out by the Galois representation $\saint \rho_{f,\l}$ is ramified only at $\ell$. Therefore, we can check that the discriminant of the polynomial $F(x) \in \Q[x]$ is of the form
\[ \pm \ell^n M^2 \]
for some $M \in \Q^*$. Even better, we can compute the maximal order of the field $K = \Q[x] / F(x)$ whose Galois closure is $L$ and check that its discriminant is, up to sign, a power of $\ell$. Since a number field ramifies at the same primes as its Galois closure, this proves that the decomposition field $L$ of $F(x)$ is ramified only at $\ell$, as expected. If the coefficients of $F(x)$ are too horrible for that, we can apply this check on $F_r(x)$ instead.

\item Since Galois representations attached to modular forms are odd, the image of complex conjugation by these representations is an involutory matrix in $\GL_2(\F_\l)$ of determinant $-1$, hence similar to $\left[ \begin{smallmatrix} 1 & 0 \\ 0 & -1 \end{smallmatrix} \right]$ if $\ell \geqslant 2$. This means that the polynomial $F(x)$ of degree $\ell^2-1$ computed by the algorithm should have exactly $\ell-1$ roots in $\R$, which can be checked numerically, and that the sign of its discriminant should be $(-1)^{\ell (\ell-1)/2}$, which can be checked exactly.

\item The fact that the resolvents $\Gamma_C(x)$ computed by the Dokchitsers' method and used to identify the image of Frobenius elements seem to have integer (and not just complex) coefficients hints that $\Gal(L/\Q)$ is indeed isomorphic to a subgroup of $\GLFl$, so that the number field $L$ is indeed a number field cut out by a Galois representation, and that the Galois action on $V_{f,\l} \subset J_1(\ell)[\ell]$ is linear. Again, we can replace $F(x)$ with $F_r(x)$ and $\GLFl$ with $\GLFl/S_r$ to ease computation.

\item The fact that the approximations $F_i(x)$ of the polynomials $\saint F_i(x)$ computed by regrouping the roots of $F(x)$ along their $S$-orbits for the various subgroups $S \subseteq \Flx$ considered during the polynomial reduction process (cf. section \ref{Polred}) seem to have rational coefficients with common denominator dividing the one of $F(x)$ also hints that the coefficients of these polynomials have been correctly identified as rational numbers, that $\Gal(L/\Q)$ is indeed isomorphic to a subgroup of $\GLFl$, and that the Galois action on the roots of $F(x)$ is the expected one.

\item Finally, we can check that the values $a_p \bmod \mathfrak{l}$ obtained by the algorithm for a few small primes $p$ are correct, by comparing them with the ones computed by ``classical'' methods such as based on modular symbol-based ones. 

\end{itemize}

We will now present a method to formally prove rigorously  our computations, while keeping the amount of required extra computations to a minimum.

\subsection{A certification algorithm}

We keep the notations of section \ref{Polred}: we fix a prime $\ell \geqslant 5$, and we let $r \in \N$ be such that $\ell-1 = 2^r m$  for some odd $m \in \N$, so that we have the filtration
\[ \Flx = S_0 \underset{2}{\supset} S_1 \underset{2}{\supset} \cdots \underset{2}{\supset} S_r = S \]
with $\# S_r$ odd and $[S_{i} : S_{i+1}] = 2$ for all $i$. Let $V = \Fl^2-\{0\}$, the vector plane minus the origin, on which $\GLFl$ acts transitively, and let $V_i = V / S_i$, so that we have a natural transitive action of $\GLFl/S_i$ on $V_i$. We denote by $\xymatrix{ \pi_i \colon V_{i+1} \ar@{->>}[r] & V_i}$ the natural projection, and we note for future reference that each element of $\GLFl/S_i$ has a well-defined trace in $\Fl/S_i$, as well as a well-defined determinant in $\Flx/S_i^2$, where
\[ S_i^2 = \{s^2, s \in S_i \} = \left\{ \begin{array}{ll} S_{i+1}, & \text{ if } i<r, \\ S_i, & \text{ if } i=r. \end{array} \right. \]

For each $0 \leqslant i \leqslant r$, we have constructed a monic, irreducible polynomial $F_i(x)~\in~\Z[x]$ of degree $2^i (\ell+1)$. Let $K_i$ be the root field of $F_i(x)$, let $L_i$ be its Galois closure. We have that $K_{i+1}$ is a quadratic extension of $K_i$, generated by the square root of some explicitly known integral primitive element $\delta_i$ of $K_i$, as this is a by-product of the reduction process presented in section \ref{Polred}. 

For each $i$, let $Z_i \subset \C$ denote the set of complex roots of $F_i(x)$. As noted in section \ref{Polred}, we have an indexation of $Z_i$ by $V_i$, which we denote by $\theta_i \colon Z_i \overset{\sim}{\longrightarrow} V_i$. Via these indexations, the Galois action on the $Z_i$ should be ``linear'', but we do not know that yet.

Besides, by construction of the $F_i(x)$, for each root $z \in Z_{i+1}$ there exists another root $z' \in Z_{i+1}$ such that $z + z'$ is extremely close to a root of $F_i(x)$. We can check numerically that each root of $F_i(x)$ is the sum of two roots of $F_{i+1}(x)$ in a unique way, whence 2-to-1 projections map $\xymatrix{ \varpi_i \colon Z_{i+1} \ar@{->>}[r] & Z_i}$ such that
\[ z \approx \sum_{\substack{z' \in Z_{i+1} \\ \varpi_i(z')=z}} z' \]
for all $z \in Z_i$. We can check that these approximate identities are in fact exact, i.e.
\begin{equation}
z = \sum_{\substack{z' \in Z_{i+1} \\ \varpi_i(z')=z}} z', \label{rac_trace} \tag{T}
\end{equation}
by computing rigorously\footnote{Here and in what follows, by \emph{rigorously} we mean by the use of exact methods such as resultants, as opposed to the expansion of the product over a non-exact field followed by the identification of the coefficients as elements of $\Z$ or $\Q$.} for each $i$ the polynomial
\[ \prod_{I \in \binom{Z_i}{2}} \left (x-\sum_{z \in I} z \right) \in \Z[x], \]
where $\binom{Z_i}{2}$ denotes the set of $2$-element subsets of $Z_i$, and by checking that $F_i(x)$ divides this polynomial and that the complex roots match as expected. We can then also check numerically that the diagram
\begin{equation}
\begin{gathered}
\xymatrix{
Z_{i+1} \ar@{->>}[d]_{\varpi_i} \ar[r]^{\sim}_{\theta_{i+1}} & V_{i+1} \ar@{->>}[d]^{\pi_i} \\ Z_i \ar[r]^{\sim}_{\theta_i} & V_i
}
\end{gathered} \label{diag_Pi} \tag{$\Pi$} \end{equation}
commutes for each $i$, as expected. This proves that the projections $\pi_i$ are Galois-equivariant. 

What we want to prove is that there exists a compatible\footnote{Here and in what follows, by \emph{compatible} we mean compatible with the natural projections from objects at level $i+1$ to objects at level $i$.} system of isomorphisms between $\Gal(L_i/\Q)$ and $\GLFl/S_i$, such that the Galois action on $Z_i$ is equivalent via our bijections $\theta_i$ to the natural action of $\GLFl/S_i$ on $V_i$, so that the diagram
\[ \xymatrix @R=5pc @C=3pc{
    \Gal(L_i/\Q) \ar[rr] \ar[dd] && \Sym(Z_{i}) \ar[dd] |!{[dl];[dr]}\hole \\
    & \Gal(L_{i+1}/\Q) \ar@{->>}[ul] \ar[rr] \ar[dd] && \Sym(Z_{i+1},\varpi_i) \ar[dd] \ar[ul] \\
    \GLFl/S_i \ar[rr] |!{[ur];[dr]}\hole && \Sym(V_i) \\
    & \GLFl/{S_{i+1}} \ar@{->>}[ul] \ar[rr] && \Sym(V_{i+1},\pi_i) \ar[ul] \\
} \]
commutes for all $i$, where the vertical arrows are isomorphisms, $\Sym(V_{i+1},\pi_i)$ denotes the group of permutations of $V_{i+1}$ that admit the fibres of the projection $\pi_i$ as a block system, and similarly for $\Sym(Z_{i+1},\varpi_i)$.

Furthermore, we also want to prove that for all $i$, the Galois action on $Z_i$ affords a quotient Galois representation $\rho_{f,\l}^{S_i}$ which is equivalent to $\saint \rho_{f,\l}^{S_i}$. For brevity, we will then say that the polynomials $F_i(x)$ \emph{correspond} to $\saint \rho_{f,\l}$.

\bigskip

We will present two methods to rigorously prove that our polynomials $F_i(x)$ correspond to a Galois representation $\saint \rho$, the second one being more efficient but unfortunately much more complicated then the first one. We will then finally show how to prove that $\saint \rho \sim \saint \rho_{f,\l}$.

Both methods require that we first check that $F_0(x)$ indeed corresponds to $\saint \rho_{f,\l}^{\text{proj}}$, so we start by showing how this can be done.

\subsection{Certification of the projective representation}

\subsubsection{Certification of the Galois group of $F_0(x)$}\label{GalProj}

We thus begin with the polynomial $F_0(x)$, which ought to correspond to the projective Galois representation $\saint \rho_{f,\l}^{\text{proj}}$. The first thing to do is to make sure that this polynomial does define a projective Galois representation, by proving that there exists an indexation of $Z_0$ by $\P^1(\Fl)$ such that $\Gal(L_0/\Q)$ is permutation isomorphic to a subgroup of $\PGLFl$ acting of $\P^1(\Fl)$. Since by assumption $f$ has level $N=1$ and is not exceptional mod $\l$, we actually expect $\Gal(L_0/\Q)$ to be isomorphic to the whole of $\PGL_2(\Fl)$.

In principle, we could prove this by computing the polynomial
\[ R_4(x) = \prod_{\substack{z_1,z_2,z_3,z_4 \in Z_0 \\ \text{pairwise distinct}}} \left(x-\sum_{n=1}^4 \lambda_n z_n \right) \in \Z[x] \]
by rigorous methods (e.g. resultants), and by checking how it factors over $\Q$. Here, the $\lambda_n$ are fixed integers chosen so that $R_4(x)$ is squarefree, so that $R_4(x)$ monitors the action of Galois on quadruplets of roots of $F_0(x)$. The point is that a permutation of $\P^1\Fl$ comes from $\PGLFl$ if and only if it preserves cross-ratios, and this should become apparent in the factorisation of $R_4(x)$.

However, the degree of $R_4(x)$ is approximately $\ell^4$, which is quite large for $\ell=31$, not to mention that since the parameters $\lambda_n$ must necessarily be distinct, the coefficients of $R_4(x)$ will be huge. As a result, computing $R_4(x)$ would be too slow in practice.

We can instead compute the polynomial \[ R_{4,\text{sym}}(x) = \prod_{I \in \binom{Z_0}{4}}\left( x - \sum_{z \in I} z \right) \in \Z[x], \]
where $\binom{Z_0}{4}$ denotes the set of sets of roots of $F_0(x)$ of cardinal $4$. This polynomial monitors the action of  Galois on \emph{unordered} quadruplets of roots of $F_0(x)$, and compared to $R_4(x)$, its degree is $24$ times smaller, and its coefficients are much smaller, so that computing it is much more amenable. It turns out that the way $R_{4,\text{sym}}(x)$ factors is enough to indicate that $\Gal(L_0/\Q)$ is a subgroup of $\PGLFl$ in most cases.

To make this claim more precise, let us fix some notation: we let $k$ be a field\footnote{We have $k = \F_\ell$ in mind, but we would like to make general statements.} of characteristic different from $2$, and $H$ to be the so-called \emph{anharmonic group}, that is to say the group of permutations of $\P^1(k)$ generated by $\lambda \mapsto 1-\lambda$ and $\lambda \mapsto 1/\lambda$. It is well-known that $H \simeq \mathfrak{S}_3$ is the stabiliser of the set $\{\infty,0,1 \}$ for the action of $\PGL_2(k)$ on $\P^1(k)$, and that if $(a,b,c,d) \in \P^1(k)^4$ is a quadruplet of pairwise distinct points, then the cross-ratios of all possible 24 permutations of this quadruplet forms an orbit under $H$. Besides, since the fibres of the map
\[ \begin{array}{ccc} j \colon \P^1(k)\setminus\{\infty,0,1\} & \longrightarrow & k \\
\lambda & \longmapsto & 256  \frac{(1-\lambda+\lambda^2)^3}{\lambda^2 (1-\lambda)^2} \end{array} \]
are precisely the $H$-orbits\footnote{This is because $j(\lambda)$ is the $j$-invariant of the Legendre curve $y^2=x(x-1)(x-\lambda)$, so that the map $j$ we define is the projection from the modular curve $X(2)$ (identified to the $\lambda$-line via Legendre curves) to $X(1)$ (identified to the $j$-line), and because $H$ is the Galois group of the covering  $X(2) \longrightarrow X(1)$ under these identifications. The author thanks S. Siksek for bringing this to his attention.}, the composition the cross-ratio with $j$ yields a well-defined ``unordered cross-ratio" map $u \colon \binom{\P^1(k)}{4} \longrightarrow  k$.

\begin{thm}
If $\ell \neq 5$, then the permutations of $\P^1(\Fl)$ that preserve the unordered cross-ratio map $u$ are precisely the ones that come from $\PGLFl$.
\end{thm}

\begin{proof}
If $\ell \leqslant 3$, then every permutation of $\P^1(\Fl)$ comes from $\PGLFl$ and so there is nothing to prove. We may therefore assume that $\ell \geqslant 7$. But then the anharmonic group $H$ does not act transitively on $\P^1(\Fl)\setminus \{\infty,0,1\}$, and so $u \colon \binom{\P^1(k)}{4} \longrightarrow  k$ is not a constant map. As a result, its stabiliser in $\mathfrak{S}_{\ell+1}$ is a strict subgroup $S < \mathfrak{S}_{\ell+1}$ which clearly contains $\PGLFl$. But $\PGLFl$ is a maximal subgroup of $\mathfrak{S}_{\ell+1}$ according to the following theorem, whence the result.
\end{proof}

\begin{thm}
Let $\ell \in \N$ be a prime. The permutation group $\PGLFl$ of $\P^1(\F_\ell)$ is a maximal subgroup of the symmetric group $\mathfrak{S}_{\ell+1}$.
\end{thm}

\begin{proof}
We may assume that $\ell \neq 2$. Suppose that there is a group $X$ such that $\PGLFl~<~X~<~\mathfrak{S}_{\ell+1}$.
Then $X$ is at least 3-transitive. By looking through the list of 2-transitive finite permutation groups given in section 7.7 of \cite{DixonMortimer}, it can be derived that the $3$-transitive finite permutation groups are the following:
\begin{itemize}
\item the projective semilinear groups $G$ with $\PSL_2(\F_q) \leqslant G \leqslant \operatorname{P \Gamma L}_2(\F_q)$, where $q$ is a power of a prime $p$ and $G \not \leqslant \operatorname{P \Sigma L}_2(\F_q)$ if $p \neq 2$, degree $q+1$,
\item the affine groups $\operatorname{AGL}_n(\F_2) = \F_2^n \rtimes \GL_n(\F_2)$, degree $2^n$,
\item the group $\F_2^4 \rtimes \mathfrak{A_7}$, degree 16,
\item the Mathieu groups $M_{11}$, $M_{12}$, $M_{22}$, $\Aut(M_{22})$, $M_{23}$ and $M_{24}$, respective degrees $11$ or $12$, $12$, $22$, $22$, $23$, $24$,
\item the alternating groups $\mathfrak{A}_n$ ($n \geqslant 5$), degree $n$,
\item and the symmetric groups $\mathfrak{S}_n$  ($n \geqslant 3$), degree $n$,
\end{itemize}
where $\operatorname{P \Gamma L}_2(k)$ (resp. $\operatorname{P \Sigma L}_2(k)$) denotes the permutation group of $\P^1(k)$ generated by $\PGL_2(k)$ (resp. $\PSL_2(k)$) and by the automorphisms of the ground field $k$.

As we want degree $\ell+1$ with $\ell$ prime, this only leaves $\operatorname{AGL}_n(\F_2)$, $M_{11}$, $M_{12}$, $M_{24}$ and $\mathfrak{A}_{\ell+1}$ as candidates for $X$. However, these groups are all perfect, so they all act by even permutations and thus cannot contain $\PGLFl$.
\end{proof}

As far as we are concerned, the main consequence of this is that it is enough to see how $R_{4,\text{sym}}(x)$ factors to prove that $\Gal(L_0/\Q)$ is permutation isomorphic to a subgroup of $\PGLFl$, and this is a stark improvement compared to working with $R_4(x)$, whose degree is 24 times larger.

\bigskip

The computer algebra package \cite{Magma} contains two functions named respectively \verb?GaloisGroup? and \verb?GaloisProof? whose aim is to compute Galois groups by the algorithm described in \cite{Kluners}. The former, when supplied with an irreducible polynomial $f(x) \in \Z[x]$ and a prime number $v \in \N$, tries to guess the Galois group of $f(x)$ as a permutation group acting on the $v$-adic roots of $f(x)$, albeit non-rigorously; in order to get a certification of this result, it is necessary to then apply the latter function.

In our case, if we pick a prime $p \in \N$ such that $F_0(x)$ is irreducible\footnote{Such a prime should exist and should not be too hard to come by, as a non-negligible proportion $\left( \frac{\varphi(\ell+1)}{2(\ell+1)}, \text{ to be precise}\right)$ of elements of $\PGLFl$ act as $(\ell+1)$-cycles on $\P^1(\Fl)$.} mod $p$, then when we call \verb?GaloisGroup? on $\big(F_0(x),p)$, it only takes a few seconds (even for $\ell=31$) for \cite{Magma} to return a guess for $\Gal(L_0/\Q)$, thanks to the efficiency of \cite{Kluners}, the fact that $F_0(x)$ is of degree only $\ell+1$ and has been \verb?polred?ed, and to the non-trivial information provided by the cyclic action of the Frobenius at $p$ on the $p$-adic roots of $F_0(x)$. However, as explained above, this is not rigorous, so we then call \verb?GaloisProof?, which forces \cite{Magma} to compute and factor $R_{4,\text{sym}}(x)$ rigorously so as to verify the output of \verb?GaloisGroup?. This takes of course much longer (up to $4$ days for $\ell=31$), and in fact this is by far the most time-consuming part of the whole certification process of our polynomials, at least for large $\ell$.

We then check explicitly that this Galois group is permutation-isomorphic to $\PGL_2(\Fl)$ acting on $\P^1(\Fl)$. We fix such an isomorphism\footnote{Note that as every automorphism of $\PGLFl$ is interior, our isomorphism must be the ``right'' one.}, and we will use it to identify $\Gal(L_0/\Q)$ with $\PGL_2(\Fl)$ from now on. This yields a bijection $\theta_0$ between the roots of $F_0(x)$ in $\overline{\Q_p}$ and $\P^1(\Fl)$ which makes the Galois action equivalent to the natural action of $\PGL_2(\Fl)$ on $\P^1(\Fl)$.

\subsubsection{Correctness of the projective representation}

Now that we have made sure that the Galois action on the roots of $F_0(x)$ does define a projective representation
\[ \xymatrix{\rho^{\text{proj}} \colon G_\Q \ar@{->>}[r] & \Gal(L_0/\Q) \ar@{=}[r] & \PGL_2(\Fl)}, \]
we want to prove that this representation is isomorphic to $\saint \rho^{\text{proj}}_{f,\l}$ as expected. For this, we use the following result from \cite[section 2]{Bos}:

\begin{thm}
Let $\saint \pi \colon G_\Q \longrightarrow \PGL_2(\Fl)$ be an irreducible projective mod $\ell$ Galois representation, where $\ell \geqslant 3$. Let $H < \PGL_2(\Fl)$ be the stabiliser of a point of $\P^1(\Fl)$, and let $K = \overline{\Q}^{\saint \pi^{-1}(H)}$ be the corresponding number field. If $K$ has exactly two real places, and if there exists an integer $k \geqslant 3$ such that
\[ \disc K = \pm \ell^{k+\ell-2}, \]
then there exists a newform $f \in S_k(1)$ and a prime $\l$ of $\overline \Q$ above $\ell$ such that
\[ \saint \pi \sim \saint \rho^{\text{proj}}_{f,\l}. \] 
\end{thm}

\begin{proof}[Sketch of proof]
By assumption, the image of complex conjugation by $\saint \pi$ is a non-trivial matrix which is diagonalisable over $\Fl$, and so $\saint \pi$ is absolutely irreducible.
 The idea is then that $\saint \pi$ can be lifted to a linear representation
\[ \saint \rho \colon G_\Q \longrightarrow \GL_2(\overline{\Fl}) \]
which, just like $\saint \pi$, is absolutely irreducible, odd, and ramifies only at $\ell$. Serre's modularity conjecture (cf. \cite{KW}) then applies and shows that $\saint \rho$ is modular, say $\saint \rho \sim \saint \rho_{f,\l}$ for some newform $f \in S_{k_{\saint \rho}}(N_{\saint \rho},\varepsilon_{\saint \rho})$ and some prime $\l$ of $\overline \Q$ above $\ell$. Then, since $\saint \rho$ ramifies only at $\ell$, its Artin conductor is a power of $\ell$, so $\saint \rho$ comes from a form $f$ of level $N_{\saint \rho}=1$. Finally, if the lift $\saint \rho$ is chosen so that the weight $k_{\saint \rho}$ of $f$ is minimal, then \cite[theorem 3]{MoonTaguchi} gives a formula for the $\ell$-adic valuation of the discriminant of the Galois number field cut out by $\saint \rho$, which by J. Bosman's work boils down to
\[ \disc K = \pm \ell^{k_{\saint \rho}+\ell-2}. \]
\end{proof}

Thus, in order to prove that $\rho^{\text{proj}} \sim \saint \rho^{\text{proj}}_{f,\l}$, all we have to do is count the real roots of $F_0(x)$, which can be done by using Sturm's method (cf. \cite[chapter XI, theorem 2.7]{LangAlgebra}), and check that the discriminant of the root field $K^{\text{proj}} = \Q[x]/F_0(x)$ is $\pm \ell^{k+\ell-2}$, which is a piece of cake for \cite{gp}. If $k$ is such that $\dim S_k(1)=1$, e.g. $k \leqslant 22$, then this is enough to conclude that $\rho^{\text{proj}} \sim \saint \rho^{\text{proj}}_{f,\l}$, as the coefficients of $f$ then are rational so that the choice of the prime $\l$ lying above $\ell$ does not matter.

In the case $\dim S_k(1)>1$, we can check that the newforms in $S_k(1)$ are all conjugate under Galois as predicted by Maeda's conjecture, and so we only have to make sure that $\rho^{\text{proj}} \sim \saint \rho^{\text{proj}}_{f,\l}$ for the right prime $\l$ above $\ell$. For instance, in the case $\ell=31,k=24$, we have that $S_{24}(1)$ has dimension $2$ and is spanned by the two conjugates of a newform $f_{24}=\sum_{n \geqslant 1} \tau_{24}(n) q^n$ whose eigenvalues lie in a quadratic field; since $31$ splits in this field, say $31 = \l_1 \l_2$, we know that $\rho^{\text{proj}}$ is equivalent either to $\saint \rho^{\text{proj}}_{f_{24},\l_1}$ or to $\saint \rho^{\text{proj}}_{f_{24},\l_2}$. In order to tell which, we pick a small prime $p \in \N$ such that $F_0(x)$ is squarefree mod $p$ (in particular $p \neq \ell$), and such that $\tau_{24}(p) \equiv 0 \bmod \l_1$ but $\tau_{24}(p) \not \equiv 0 \bmod \l_2$ (the opposite would do too). Since an element of $\PGL_2(\Fl)$ is of order $2$ if and only if it has trace $0$, looking at the factorisation of $F_0 \bmod p$ allows us to tell $\l_1$ and $\l_2$ apart: if $F_0(x) \bmod p$ splits into linear and quadratic factors but does not split completely, then it is associated to $\saint \rho^{\text{proj}}_{f_{24},\l_1}$, else it is associated to $\saint \rho^{\text{proj}}_{f_{24},\l_2}$.

\subsection{Two approaches to the certification of the Galois groups of the $F_i(x)$}

In principle, we could simply ask again \cite{Magma} to determine the Galois group of the $F_i(x)$, as we did above for $F_0(x)$. However, the permutation groups $\GLFl/S_i$ are not characterised as nicely as $\PGL_2(\Fl)$, which can be defined as the group of permutations of $\P^1(\Fl)$ that preserve cross-ratios. As a result, Magma would have to rely on much more involved group-algorithmic methods, which would make the computation much slower\footnote{In fact, the \texttt{GaloisGroup} function can still make the right guess pretty quickly, but this guess must then be proved by calling the \texttt{GaloisProof} function, and this is far too slow because of the degree of the polynomials $F_i(x)$ for $i >0$. For instance, for $\Delta$ mod $19$, it takes \cite{Magma} four days to laboriously manage to certify that the Galois group of $F_1(x)$ is permutation isomorphic to $\GL_2(\F_{19})/\F_{19}^{*2}$. By comparison, the method presented in section \ref{Galois_cohom} below merely takes a few minutes.}. We are going to present methods which require much less computation, and which also yield proofs that are more human-readable.

We are actually going to present two methods to exhibit a permutation isomorphism between the Galois group of $F_r(x)$ and $\GLFl/S_r$ acting naturally on $V_r=V/S_r$. The first one, which we present in section \ref{Galois_geom}, is the more natural one, and is entirely due to the reviewer of an older version of this article; the author wishes to thank him profusely for this. Unfortunately, it leads to computations which, albeit not as slow as a blunt \cite{Magma} attack, still require quite a bit of computation time. The method that we will present in the next section \ref{Galois_cohom} requires much less computation time; unfortunately, it is also much more complicated to explain.

\subsection{The geometric approach}\label{Galois_geom}
 
 The method presented in this section could be used with pretty much any representation $\saint \rho \colon \Gal(\overline \K/\K) \longrightarrow \GLFl$ whose quotients $\saint \rho^{S_i}$ are surjective, where $\K$ can be any field\footnote{We have a number field in mind.} we can perform computations such as polynomial factorisation with.
 
In this section, we thus suppose that we have a collection of irreducible polynomials $F_i(x) \in \K[x]$, $0 \leqslant i \leqslant r$, which ought to correspond to such a Galois representation $\saint \rho$. We also suppose that the $F_i(x)$ split completely in some extension\footnote{We have $\Omega=\C$ or some finite extension of $\Q_p$ in mind.}$\Omega$ of $\K$, and that we have conjectured a compatible system a bijections $(\theta_i)_{0 \leqslant i \leqslant r}$ between the roots of $F_i(x)$ in $\Omega$ and the $V_i$ such that we expect the Galois action on the roots of $F_i(x)$ to be permutation isomorphic to $\GLFl/S_i$ acting naturally on $V_i$, and such that the relations \eqref{diag_Pi} and \eqref{rac_trace} hold between the roots of $F_i(x)$ and those of $F_{i+1}(x)$ for all $i < r$. For each $i$, we identify the set of roots of $F_i(x)$ with $V_i$, and the Galois group of $F_i(x)$ with a permutation group of $V_i$ for each $i$; the projections $\xymatrix{\pi_i \colon V_{i+1} \ar@{->>}[r] & V_i}$ are then Galois-equivariant. Finally, we also assume that we have managed to prove that the Galois group of $F_0(x)$ is indeed permutation isomorphic to $\PGLFl$ via $\theta_0$ by a method similar to the one described in section \ref{GalProj} above. We may thus identify the Galois group of $F_0(x)$ with $\PGLFl$.
 
Our goal is to prove that the Galois group of $F_r(x)$ is contained in $\GLFl/S_r$. The key idea of the method presented in this section is to prove that its action on $V_r$ is ``linear''. However, as the addition of vectors does not descend to a well-defined operation on $V_r$, this is not completely straightforward.

\bigskip

To begin with, the relation \eqref{rac_trace} tells us that the Galois group of $F_r(x)$ is a subgroup of the wreath product $\Sym(\Flx/S_r) \wr \PGLFl$. We first want to prove that it is actually a subgroup of $(\Flx/S_r) \wr \PGLFl$, in other words that the action of Galois commutes with scalar multiplication.

Clearly, it is enough to prove that Galois commutes with the scalar multiplication by a generator $\varepsilon$ of $\Flx/S_r$. To do so, we compute by interpolation a polynomial $\widetilde E(x) \in \Omega[x]$ which, for all $v \in V_r$, maps the root of $F_r(x)$ indexed by $v$ to the root indexed by $\varepsilon \cdot v$. We then try to identify the coefficients of this polynomial as approximations of elements of $\K$, whence a polynomial $E(x) \in \K[x]$. If $E(x)$ indeed approximately maps the root indexed by $v$ to the one indexed by $\varepsilon \cdot v$ for all $v \in V_r$ and if $F_r(x)$ divides $F_r \circ E(x)$, this proves that the Galois action commutes with scalar multiplication on $V_r$.

We expect this approach to succeed since multiplication by $\varepsilon$ indeed defines an automorphism not only of the splitting field but also of the root field $\K[x]/\saint F_r(x)$ of $\saint F_r(x)$. Besides, interpolating over the roots of $F_r(x)$ amounts to solving a linear system whose determinant is the discriminant of $F_r(x)$, so that the coefficients of $E(x)$ should not be too difficult to identify if the ones of $F_r(x)$ are nice. In practice, with our \verb?polred?ed polynomial $F_r(x) \in \Z[x]$, it indeed takes just a few seconds to compute $E(x) \in \Q[x]$ and to check that $F_r \circ E(x) \equiv 0 \bmod F_r(x)$.

\bigskip

We may thus assume henceforth that the Galois group of $F_r(x)$ is contained in $P = (\Flx/S_r) \wr \PGLFl$. Let us consider, for all triples $(L_1,L_2,M) \in \P^1(\Fl)^3$ of pairwise distinct vector lines in $\F_\ell^2$, the map
\[ t_{L_1,L_2,M} \colon L_1 \longrightarrow L_2 \]
that sends a point $x \in L_1$ to the intersection of $L_2$ and of the line through $x$ that is parallel to $M$ (cf. the figure below).

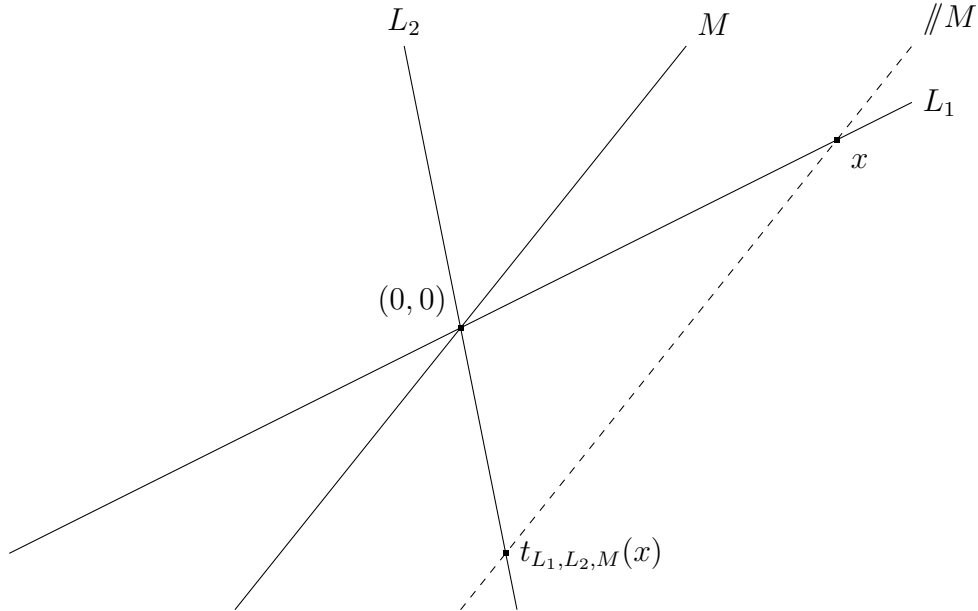
\begin{center}
\begin{tikzpicture}[scale=1.5]
\coordinate [label=right:$L_1$] (L1end) at (4,2);
\coordinate [label=above:$L_2$] (L2end) at (-0.5,2.5);
\coordinate [label=above right:$M$] (Mend) at (2,2.5);
\coordinate [label=above right:$\sslash M$] (Mparend) at (4,2.5);
\draw[name path=L1] (-4,-2) -- (4,2);
\draw[name path=L2] (0.5,-2.5) -- (-0.5,2.5);
\draw[name path=M] (-2,-2.5) -- (2,2.5);
\draw[name path=Mpar,style=dashed] (0,-2.5) -- (Mparend);
\path [name intersections={of=L1 and L2,by=O}];
\node [fill=black,inner sep=1pt,label=170:{$(0,0)$}] at (O) {};
\path [name intersections={of=L1 and Mpar,by=x}];
\node [fill=black,inner sep=1pt,label=-45:$x$] at (x) {};
\path [name intersections={of=L2 and Mpar,by=y}];
\node [fill=black,inner sep=1pt,label=0:$t_{L_1,L_2,M}(x)$] at (y) {};
\end{tikzpicture}
\end{center}

Clearly, for all $S \leqslant \Flx$, this map descends to a map
\[ t_{L_1,L_2,M}^S \colon L_1/S \longrightarrow L_2/S. \]

If we now let $X$ denote the set of triples $(v_1,v_2,M) \in V_r \times V_r \times \P^1(\Fl)$ such that the line $L_1$ spanned by $v_1$, the line $L_2$ spanned by $v_2$, and the line $M$ are pairwise distinct, we can define another map
\[ \Lambda \colon X \longrightarrow \Flx/S_r \]
by sending $(v_1,v_2,M)$ to the unique scalar $\lambda \in \Flx/S_r$ such that $v_2 = \lambda \cdot  t_{L_1,L_2,M}^S(v_1)$. The group $\GLFl/S_r$ acts diagonally on $X$, and it is clear that $\Lambda$ is invariant under this action. Conversely, we have the following:

\begin{lem}
Let $\sigma \in P$. If $\Lambda(\sigma \cdot v_1, \sigma \cdot v_2, \sigma \cdot M) = \Lambda(v_1,v_2,M)$ for all $(v_1,v_2,M) \in X$, then $\sigma \in \GLFl/S_r$.
\end{lem}

\begin{proof}
Let $\sigma \in P$ be such an element, and let $g \in \GLFl/S_r$ have the same image in $\PGLFl$ as $\sigma$. Then $\sigma' = g^{-1} \sigma$ lies in $(\Flx/S_r)^{\P^1(\Fl)}$ and leaves $\Lambda$ invariant, and so in fact lies in the diagonal $\Flx/S_r$. It follows that $\sigma \in \GLFl/S_r$.
\end{proof}

As a result, all we need to do is check that $\Lambda$ is invariant under Galois. This leads us to the resolvent
\[ R(x)=\prod_{(\alpha_1,\alpha_2,\alpha_3) \in Z}\left(x- \sum_{i=1}^3 \lambda_i \alpha_i \right) \in \K[x], \]
where the $\lambda_i \in \Z$ are parameters chosen so that $R(x)$ is squarefree, and $Z$ is the set of triples $(\alpha_1,\alpha_2,\alpha_3)$ with $\alpha_3$ a root $F_0(x)$, $\alpha_1$,$\alpha_2$ roots of $F_r(x)$, and $\alpha_1,\alpha_2$ and $\alpha_3$ corresponding to three distinct roots of $F_0(x)$ under the correspondence \eqref{rac_trace}. If we can compute $R(x)$ rigorously and prove that it factors along the fibres of $\Lambda$, then we have proved that the Galois group of $F_r(x)$ is contained in $\GLFl/S_r$.

Unfortunately, just like the resolvent $R_4(x)$ from section \ref{GalProj}, the resolvent $R(x)$ would take a lot of time to compute in our case. Indeed, its degree is $2^{2r}(\ell^3-\ell)$, and for us this is too much: we have $\ell \leqslant 31$ in mind, but even if we restricted ourselves to the primes $\ell \equiv -1 \bmod 4$ so that $r=1$ so as to get an asymptotic $\deg R(x)=O(\ell^3)$ which is better than the degrees $\deg R_{4}(x)=O(\ell^4)$ of the resolvents considered in section \ref{GalProj}, we would still have
\[ \deg R(x) \gg \deg R_{4,\text{sym}}(x), \]
due to the factor $24$ in $\deg R_{4,\text{sym}}(x) = \binom{\ell+1}{4} \sim \ell^4/24$. In fact, it can easily be checked that $\deg R(x) > \deg R_{4,\text{sym}}(x)$ for all $\ell < 103$, which incidentally illustrates again how useful switching from $R_4(x)$ to $R_{4,\text{sym}}(x)$ was in section \ref{GalProj}.

As certifying the Galois group of $F_0(x)$ thanks to the resolvent $R_{4,\text{sym}}(x)$ already took up to 4 days for $\ell=31$, this is a real problem. For this reason, we introduce another method to certify the Galois group of the $F_i(x)$ in the next section. This other method is much more complicated, but the computation time it requires is almost negligible compared to the time needed to certify the Galois group of $F_0(x)$, at least when $r \leqslant 2$.

\subsection{The group cohomology approach}\label{Galois_cohom}

Just like the method presented in the previous section, the method that we are now going to introduce could be applied to a more general framework than the case of modular Galois representations attached to forms of level $1$. It is not as general as the previous one though, in that it requires working with representations $\saint \rho \colon \Gal(\overline \Q/\Q) \longrightarrow \GLFl$ whose image contains $\SLFl$ and whose determinant is still an odd power of the mod $\ell$ cyclotomic character. For instance, it could be used to certify Galois representation computations attached to newforms of level $\Gamma_0(N)$.

Therefore, in this section we merely suppose that we want to prove that the polynomials $F_i(x) \in \Z[x]$ correspond to such a Galois representation $\saint \rho$. In particular, this implies that $\saint \rho^{S_i}$ surjects to $\GLFl/S_i$ for all $i$. We also suppose that we have a relation of the form \eqref{rac_trace} between the roots of $F_i(x)$ and $F_{i+1}(x)$, that is to say that for all $i<r$, any root of $F_i(x)$ is the sum of precisely two roots of $F_{i+1}(x)$. However, even though we want to prove the existence of a compatible system of indexations of the sets $Z_i$ of roots of $F_i(x)$ by $V_i$ making the Galois action permutation isomorphic to the natural action of $\GLFl/S_i$, this time we do \textbf{not} suppose that we already have a candidate for such a system of indexations. Indeed, we are going to work with $p$-adic roots, whereas our algorithm \cite{Moi} returns a candidate indexation of the \emph{complex} roots of the $F_i(x)$. We therefore let $Z_i$ denote the set of roots of $F_i(x)$ in some large enough extension of $\Q_p$, where $p \in \N$ is some fixed prime. We reserve the letter $p$ for this prime from now on.

\begin{rk}
It could be argued that since Magma's function \verb?GaloisGroup? is so efficient, we could easily find a candidate for such a system of indexation of the $Z_i$ by the $V_i$ if we wanted to. However, we would still have to prove that this indexation is correct, and the method which we are going to present will involve constructing a certified system of indexations from scratch anyway.
\end{rk}

In the last steps of the method presented in this section, it will be necessary to assume that $p$ is such that $F_r(x)$ (and hence all the $F_i(x)$) is irreducible mod $p$, and it will be convenient to further assume that $p$ is rather large, say roughly the size of a machine word. Just as in the projective case, there are plenty of elements of $\GLFl/S_r$ which act as transitive cycles on $V_r$, so such a prime should not be too difficult to come by. We thus henceforth assume that $p$ is such a prime, and that we gave this $p$ as a parameter to \cite{Magma} when we certified that the Galois group of $F_0(x)$ may be identified to $\PGLFl$ as a permutation group of $Z_0$.

Finally, we let as before $K_i$ denote the root field $K_i = \Q[x]/F_i(x)$ of $F_i(x)$, \linebreak and $L_i$ denote its Galois closure, and we suppose that for each $i<R$ we know a primitive integral element $\delta_i \in K_i$ such that $K_{i+1}=K_i(\sqrt\delta_i)$.

The idea of the method which we are going to present is to first see the Galois group of $F_i(x)$ as a group extension of $\PGLFl$ of a certain kind, then to use explicit group cohomology arguments so as to establish a finite list of possibilities for this group, and next to rely on ramification arguments to eliminate all possibilities but one\footnote{This is where we need the hypothesis that $\det \saint \rho$ is a power of the mod $\ell$ cyclotomic character for our approach to have a chance to work.}. This process will rely on an induction on $i$, and will allow us to prove that $\Gal(L_i/\Q) \simeq \GLFl/S_i$ as an \emph{abstract} group. We will then prove, again by induction on $i$, that this isomorphism can be turned into an isomorphism of \emph{permutation} groups, in other words that Galois acts on $Z_i$ in the expected way. Finally, we will use the Frobenius at $p$ to determine explicitly a system of indexation of the $Z_i$ by the $V_i$ corresponding to this isomorphism.

\subsubsection{Certification of the Galois group of $F_i(x)$ as an abstract group}

Let
\[ \Q = \kappa_0 \underset2{\subsetneq} \kappa_1 \underset2{\subsetneq} \cdots \underset2{\subsetneq} \kappa_r \underset{\text{odd}}{\subseteq} \Q(\mu_\ell) \]
be the subfields of the cyclotomic extension $\Q(\mu_\ell)$ such that for all $0 \leqslant i \leqslant r$, $\Gal(\kappa_i/\Q) \simeq \Z/2^i \Z$. Thus for instance $\kappa_1 = \Q(\sqrt{\ell^*})$, where $\ell^* = \left(\frac{-1}{\ell}\right) \ell$.

\pagebreak

Consider the following assertions:

\begin{enumerate}[label=(A\arabic*)]

\item\label{A1} If $C \subseteq L_r$ is a Galois subfield of $L_r$ such that $\Gal(C/\Q) \simeq \Z/2^k\Z$ for some integer $k \leqslant r+1$, then $C$ ramifies only at $\ell$.

\item\label{A2} For each $i<r$, let $\Delta_i(x) \in \Z[x]$ be the monic minimal polynomial of $\delta_i$ over~$\Q$, and let
\[ Q_i(x) = \frac{\Res_y\big( \Delta_i(y), \Delta_i(xy) \big)}{(x-1)^{2^i(\ell+1)}} \in \Z[x]. \]
Then, for each irreducible factor $R(x)$ of $Q_i(x)$ over $\Q$, there exists an integer $j \leqslant i$ such that the field $\Q[x]/R(x)$ does not contain $\kappa_{j+1}$, whereas the algebra $\Q[x]/R(x^2)$ does contain $\kappa_{j+1}$ (as a subalgebra with unit).

\item\label{A3} For each $i<r$, there exists a prime $v \in \N$ such that $F_i(x)$ is squarefree and totally split mod $v$, but $F_{i+1}(x)$ is not.

\end{enumerate}

We do not know yet whether these assertions hold, but, we expect them to:

\begin{enumerate}
\item Since the abelianisation of $\GLFl$ is given by the determinant, if, as expected, the polynomials $F_i(x)$ have Galois group $\GLFl/S_i$ and correspond to a Galois representation whose determinant is a power of the mod $\ell$ cyclotomic character, then the maximal Abelian subextension of $L_r$ will be contained in the cyclotomic extension $\Q(\mu_\ell)$; we therefore expect \ref{A1} to hold. 

Conversely, we note that if \ref{A1} holds, then any 2-cyclic subextension field $L_r$ is contained in $\Q(\mu_{\ell^\infty})$, hence in $\Q(\mu_{\ell})$. Since $\PGL_2(\Fl)$ has a quotient $\PGL_2(\Fl)/\PSL_2(\Fl)$ of order $2$, the fields $L_i \supset L_0$ all have at least one quadratic subfield, which must then be $\kappa_1 = \Q(\sqrt{\ell^*})$, and in particular be unique. We will use this fact repeatedly to prove theorem \ref{thm_preuve} below.

\item  We expect \ref{A2} to hold, but it will make much more sense to explain why after the proof of lemma \ref{lem_Lquad} below, so we postpone the explanation to remark~\ref{Rk_whyA2}. For now, we just note that for any polynomial $P(x)= \prod_{i=1}^{n}(x-\alpha_i)$ such that $P(0) \neq 0$,
\[ \Res_y \big( P(y), P(xy) \big) = (-1)^n P(0) \prod_{i,j} \left( x- \frac{\alpha_i}{\alpha_j} \right), \]
so that
\[ \frac{\Res_y \big( P(y), P(xy) \big)}{(x-1)^n} = (-1)^n P(0) \prod_{i\neq j} \left( x- \frac{\alpha_i}{\alpha_j} \right). \]
Therefore, $Q_i(x)$ is indeed a polynomial.

\item Finally, we also expect \ref{A3} to hold: for each $i$, it suffices to consider a prime at which the Frobenius element is $\left[ \begin{smallmatrix} \varepsilon & 0 \\ 0 & \varepsilon \end{smallmatrix} \right]$ for some $\varepsilon \in S_i - S_{i+1}$. We can thus even predict that at such a prime, while $F_i(x)$ splits in linear factors, $F_{i+1}(x)$ will split in quadratic factors. 
\end{enumerate}

\bigskip

Conversely, in this subsection and the next one, we are going to prove the following result, which thus yields an efficient method to formally prove our computations:

\pagebreak

\begin{thm}\label{thm_preuve}
Assume that the assertions \ref{A1}, \ref{A2} and \ref{A3} hold. In addition, if $\ell$ is such that $r \geqslant 3$, also assume that $\kappa_{i+1} \subset L_i$ for all $2 \leqslant i < r$. Then, for all $i \leqslant r$,
\begin{enumerate}[label=(\roman*)]
\item $\Gal(L_i/\Q)$ is isomorphic to $\GL_2(\Fl)/S_i$, not only abstractly, but also as an extension of $\PGL_2(\Fl)$, and
\item there exists such an isomorphism which makes the Galois action on the roots of $F_i(x)$ equivalent to the natural action of $\GL_2(\Fl)/S_i$ on $V_i$.
\end{enumerate} 
\end{thm}

\begin{rk}
It is unfortunate that we have to make the extra assumption that $\kappa_{i+1} \subset L_i$ for all $2 \leqslant i < r$ when $r \geqslant 3$, especially as the author does not know of any computationally cheap way to check this assumption formally. Indeed, if as expected the polynomials $F_i(x)$ correspond to a Galois representation $\saint \rho$, then under the isomorphism $\Gal(L_i/\Q)\simeq \GLFl/S_i$, $\kappa_{i+1}$ corresponds to the kernel of the determinant, whereas the bigger compositum of $K_i$ with itself (cf. remark \ref{Rk_whyA2} below) corresponds to $\left\{ \left[ \begin{smallmatrix} s & 0 \\ 0 & s' \end{smallmatrix} \right] \bmod S_i \ \vert \ s,s' \in S_i \right\}$ and so does not contain $\kappa_{i+1}$, so that unfortunately one has to deal with the 3-fold compositum of $K_i$ to show that $\kappa_{i+1} \subset L_i$. The method presented in the previous section does not have this problem; on the other hand, the values of $\ell$ for which we will use theorem \ref{thm_preuve} in section \ref{Tables} are all such that $r \leqslant 2$, except for $\ell=17$ for which even the method presented in section \ref{Polred} does not suffice to reduce the polynomials $F_i(x)$ anyway.
\end{rk}

Before we start proving theorem \ref{thm_preuve}, let us indicate how the assertions \ref{A1}, \ref{A2} and \ref{A3} can be checked in practice.

\begin{enumerate}
\item Let $N$ be the product of the odd primes different from $\ell$ that ramify in $L_r$, and let $C$ be a 2-cyclic sub extension of $L_r$ of degree $2^k$, $k \leqslant r+1$. Then $C \subseteq \Q(\mu_{2^{r+3} \ell N})$, and so $\Gal\big( \Q(\mu_{2^{r+3} \ell N}) / C\big)$ is the kernel of some surjective morphism
\[ \varphi \colon \Gal\big( \Q(\mu_{2^{r+3} \ell N}) / \Q\big) \simeq (\Z/2^{r+3} \ell N\Z)^* \longrightarrow \Z/2^k\Z. \]
By Chinese remainders, we can write $\varphi = \varphi_\ell+\psi$, where
\[ \varphi_\ell \colon (\Z/\ell\Z)^* \longrightarrow \Z/2^k\Z \quad \text{and} \quad \psi \colon (\Z/2^{r+3} N\Z)^* \longrightarrow \Z/2^k\Z.\]
We then look for odd primes $v \in \N$ such that $v \equiv 1 \bmod \ell$ and $F_r(x)$ is squarefree and splits completely mod $v$. For such $v$, we have $\varphi_\ell(v)=0$ and $\varphi(v)=0$, so that $\psi(v)=0$ too. Therefore, if we can find a collection of such $v$ that spans $(\Z/2^{r+3} N\Z)^* \otimes \Z/2^{r+1}\Z$, then this proves that $\psi$ is necessarily trivial, and thus that \ref{A1} holds.

In practice, finding primes $v$ which split $F_r(x)$ completely should not be too difficult since we expect $\Gal(L_r/\Q)$ to be isomorphic to the $\GLFl/S_r$. Then, the fact that a collection of primes $v$ spans $(\Z/2^{r+3} N\Z)^* \otimes \Z/2^{r+1}\Z$ can be checked by expressing the latter group explicitly as a product of cyclic groups, by determining the image of the primes $v$ in these groups thanks to a discrete logarithm computation, and finally by computing a Smith normal form. This should all be painless, as $N$ will typically involve few prime factors, and these primes will not be very large. Note that even in the case where $r$ is large, the $(\Z/2^{r+3}\Z)^*$-part can be treated easily, since for any $a \in \N$, a subgroup of $(\Z/2^a\Z)^*$ which surjects onto $(\Z/8\Z)^*$ is necessarily the whole of $(\Z/2^a\Z)^*$.

We expect this approach to succeed, because if, as expected, $\Gal(L_r/\Q)$ is isomorphic to $\GLFl/S_r$ and the determinant of the associated Galois representation is a power of the mod $\ell$ cyclotomic character, then $L_r$ will have a unique maximal 2-cyclic subextension $C$, which has no non-trivial Abelian subextensions since
\[ \Gal(L_r/C) = \{ A \in \GLFl/S_r \ \vert \det A = 1 \} \simeq \SLFl \]
has trivial abelianisation.

Note that in the particular case of a Galois representation of level $1$, there is much less work to do: it suffices to check that the discriminant of $K_r$ is, up to a sign, a power of $\ell$.

Also note that if \ref{A1} does hold, then $L_r$ cannot actually have any subextension $C$ such that $\Gal(C/\Q) \simeq \Z/2^{r+1}\Z$, by definition of $r$.

\item We explain in remark \ref{Rk_whyA2} below why we expect $Q_i(x)$ to factor into $2^i-1$ irreducible factors of degree $2^i (\ell+1)$ and one large irreducible factor, and why \ref{A2} should be satisfied for $j=0$ for the small factors and $j=1$ for the large factor. To check that $\kappa_{j+1} \not \subset \Q[x]/R(x)$, it suffices to find a prime $v \in \N$ such that the splitting behaviour of $R(x) \bmod v$ is inconsistent with the splitting behaviour of $v$ in $\kappa_{j+1}$, for instance such that $v$ is not a square mod $\ell$ and such that $R(x) \bmod v$ is squarefree and splits into factors whose degrees are not all divisible by $2^{j+1}$. To prove that $\kappa_{j+1} \subset \Q[x]/R(x^2)$, we check that $R(x^2)$ splits into $2^{j+1}$ factors over $\kappa_{j+1}$, which means that the $\Q$-algebra $\kappa_{j+1} \otimes_\Q \big(\Q[x]/R(x^2)\big)$ has $2^{j+1}$ factors and thus that the minimal polynomial of a primitive element of $\kappa_{j+1}$ splits completely\footnote{In principle one may directly factor over $\Q[x]/R(x^2)$ the minimal polynomial of a primitive element of $\kappa_{j+1}$, but this would involve performing arithmetic in $\Q[x]/R(x^2)$ and in particular to compute an integral basis thereof, which is much slower than working over $\kappa_{j+1}$.} in $\Q[x]/R(x^2)$.

Although this is the most computation time-demanding part of the certification process, it is quite fast when $r=1$ (for $\ell=31$ it merely takes a few minutes on the author's laptop), which occurs for half of the values of $\ell$, and for $r=2$ it remains quite tractable. This is a major improvement compared to the geometric method presented in section \ref{Galois_geom}.

\item For \ref{A3}, we simply loop over primes $v \in \N$ and factor the polynomials $F_i(x)$ mod $v$ until all the couples $(i,i+1)$ have been dealt with. As explained above, such primes $v$ should not be too hard to come by.
\end{enumerate}

\bigskip

We assume henceforth that \ref{A1}, \ref{A2} and \ref{A3} hold, and proceed to the proof of part (i) of theorem \ref{thm_preuve}. Our proof consists in examining $\Gal(L_i/\Q)$ inductively for $i=1,\dots,r$. For clarity, we have divided the induction loop into six steps.

\bigskip
\noindent\framebox{\textbf{Step 1: The Galois closures are not so large}}
\bigskip

Since $K_{i+1} = K_i(\sqrt{\delta_i})$, we know that
\[ L_{i+1} = L_i\big(\sqrt{\delta_i^\sigma}, \sigma \in \Gal(L_i/\Q) \big). \]
\begin{lem}\label{lem_Lquad}
Actually, $L_{i+1} = L_i(\sqrt{\delta_i})$ is a nontrivial quadratic extension of $L_i$.
\end{lem}

\begin{proof}
According to \ref{A3}, for each $i$, there exists a rational prime that is totally split in $L_i$ but not in $L_{i+1}$, which proves that the extension $L_{i+1}/L_i$ is not trivial. Showing that it is quadratic amounts to proving that $\frac{\delta_i^\sigma}{\delta_i}$ is a square in $L_i$ for all $\sigma \in \Gal(L_i/\Q)$. To see this, pick a $\sigma \in G_\Q$ such that $\delta_i^\sigma \neq \delta_i$, so that $\Q\big(\frac{\delta_i^\sigma}{\delta_i}\big)$ is isomorphic to $\Q[x]/R(x)$ for some irreducible factor of $Q_i(x)$ over $\Q$. The polynomial $R(x^2)$ may be reducible, but in any case $\Q\big(\frac{\delta_i^\sigma}{\delta_i}\big)$ is a factor of the algebra $\Q[x]/R(x^2)$. 

We claim that $\kappa_{i+1} \subset L_i$ for all $i < r$. Indeed,

\begin{itemize}
\item for $i=0$ it follows from the fact that $\Gal(L_0/\Q)=\PGLFl$ has a quotient of order $2$ so that $L_0$ has a quadratic subfield, which can only be $\kappa_1= \Q(\sqrt{\ell^*})$ according to \ref{A1},
\item for $i=1$ (which we only need to consider when $r \geqslant 2$), it follows again from \ref{A1} and the fact that we know after one induction loop (cf. proposition \ref{Gal_L1}) that $\Gal(L_1/\Q)$ is isomorphic to $\GLFl/S_1$ and thus has a quotient isomorphic to $\Z/4\Z$ since $r \geqslant 2$,
\item and finally, for $i\geqslant 2$, this is the extra hypothesis of theorem \ref{thm_preuve} (which we thus only need when $r \geqslant 3$).
\end{itemize}

Therefore, for $j \leqslant i$ we may consider then the extension diagram
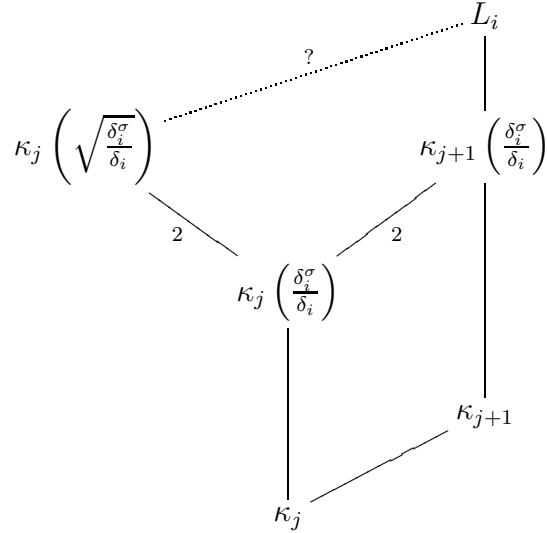
\[ \xymatrix{
&& L_i \\
\kappa_j\left(\sqrt{\frac{\delta_i^\sigma}{\delta_i}}\right) \ar@{..}[urr]^-{?} & & \kappa_{j+1}\left(\frac{\delta_i^\sigma}{\delta_i} \right) \ar@{-}[u] \\
& \kappa_j\left(\frac{\delta_i^\sigma}{\delta_i} \right) \ar@{-}[ul]^{2} \ar@{-}[ur]_{2} & \\
&& \kappa_{j+1} \ar@{-}[uu] \\
& \kappa_j \ar@{-}[uu] \ar@{-}[ur] & \\
} \]

The two extensions marked with a $2$ in this diagram are at most quadratic. We may assume that the extension $\kappa_j\left(\sqrt{\frac{\delta_i^\sigma}{\delta_i}} \right) / \kappa_j\left(\frac{\delta_i^\sigma}{\delta_i} \right)$ is not trivial, since the proof that $\sqrt{\frac{\delta_i^\sigma}{\delta_i}} \in L_i$ is over if it is. According to \ref{A2}, we may pick $j$ such that $\kappa_{j+1} \subset \Q\left(\sqrt{\frac{\delta_i^\sigma}{\delta_i}}\right)$. But then we must have
\[ \kappa_j\left(\sqrt{\frac{\delta_i^\sigma}{\delta_i}}\right) = \kappa_{j+1}\left(\frac{\delta_i^\sigma}{\delta_i} \right), \]
so that $\sqrt{\frac{\delta_i^\sigma}{\delta_i}} \in L_i$ as claimed.
\end{proof}

As a consequence, $L_{i+1} = L_i(\sqrt{\delta_i})$ and $\Gal(L_{i+1}/\Q)$ is an extension of $\Gal(L_{i}/\Q)$ by $C_2$. This extension is necessarily central, since $\Aut(\Z/2\Z)$ is trivial.

\begin{rk}\label{Rk_whyA2}
If $F_i(x)$ corresponds to $\saint \rho^{S_i}$ as expected, and if it holds that $K_i = \Q(\delta_i)$ and that $\Q\left(\frac{\delta_i^\sigma}{\delta_i}\right) = \Q(\delta_i, \delta_i^\sigma)$ (which is extremely likely), then we get an indexation of the conjugates of $\delta_i$ by $V_i$, and under the identification of $\Gal(L_i/\Q)$ with $\GLFl/S_i$ provided by $\rho^{S_i}$, the field $\Q\left(\frac{\delta_i^\sigma}{\delta_i}\right)$ corresponds by Galois theory to a conjugate of the subgroup $\{ \left[ \begin{smallmatrix} s & * \\ 0 & * \end{smallmatrix} \right] \bmod S_i, \, s \in S_i \}$ or $\{ \left[ \begin{smallmatrix} s & 0 \\ 0 & s' \end{smallmatrix} \right] \bmod S_i, \, s,s' \in S_i \}$ of $\GLFl/S_i$, depending on whether the vectors indexing $\delta_i$ and $\delta_i^\sigma$ are collinear or not. Therefore, we expect $Q_i(x)$ to split over $\Q$ into $2^i-1$ irreducible factors of degree $2^i (\ell+1)$, corresponding to the nontrivial scalar elements in $\GLFl/S_i$, plus one large irreducible factor corresponding to nonscalar elements.

Besides, in the case when the vectors indexing $\delta_i$ and $\delta_i^\sigma$ are not collinear, for $j=i$ the Galois subgroup diagram corresponding to the subfield diagram in the above proof would be
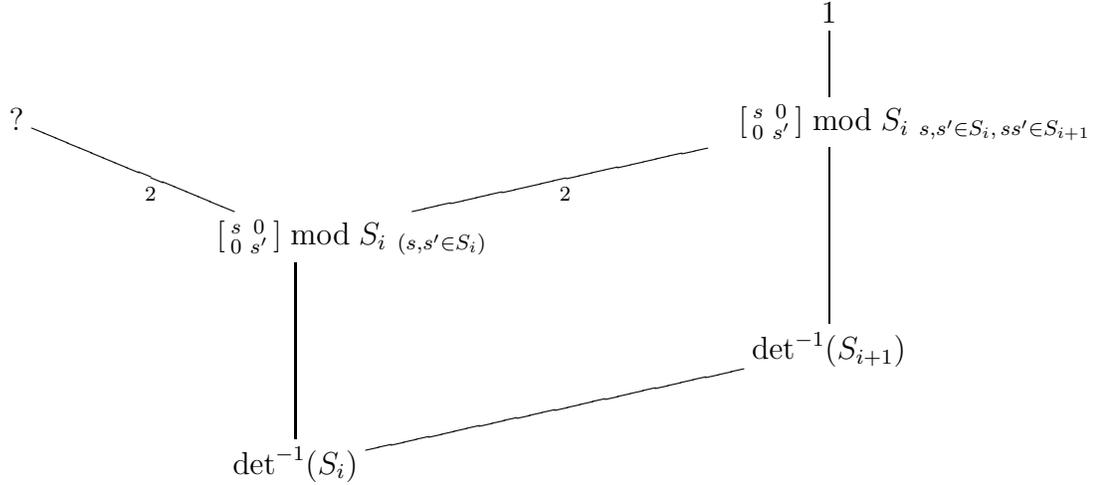
\[ \xymatrix{
&& 1 \\
? & & {\phantom{_{s,s' \in S_i, \, ss' \in S_{i+1}}}} {\left[ \begin{smallmatrix} s & 0 \\ 0 & s' \end{smallmatrix} \right] } \bmod S_i \ _{s,s' \in S_i, \, ss' \in S_{i+1}} \ar@{-}[u] \\
& {\phantom{ A_{(s,s' \in S_i)}} } {\left[ \begin{smallmatrix} s & 0 \\ 0 & s' \end{smallmatrix} \right] } \bmod S_i \ _{(s,s' \in S_i)} \ar@{-}[ul]^2 \ar@{-}[ur]_2 & \\
&& \det^{-1}(S_{i+1}) \ar@{-}[uu] \\
& \det^{-1}(S_i) \ar@{-}[uu] \ar@{-}[ur] & \\
} \]
But the group $\{ \left[ \begin{smallmatrix} s & 0 \\ 0 & s' \end{smallmatrix} \right] \bmod S_i, \, s,s' \in S_i \}$ is isomorphic to $S_i$, hence is cyclic, so the two quadratic extensions of $\Q\left(\frac{\delta_i^\sigma}{\delta_i}\right)$ marked with a $2$ in the above diagrams should coincide. We therefore expect \ref{A2} to hold for the large factor of $Q_i(x)$ with $j=i$.

Similarly, when the vectors indexing $\delta_i$ and $\delta_i^\sigma$ are collinear, we get for $j=0$ the subgroup diagram
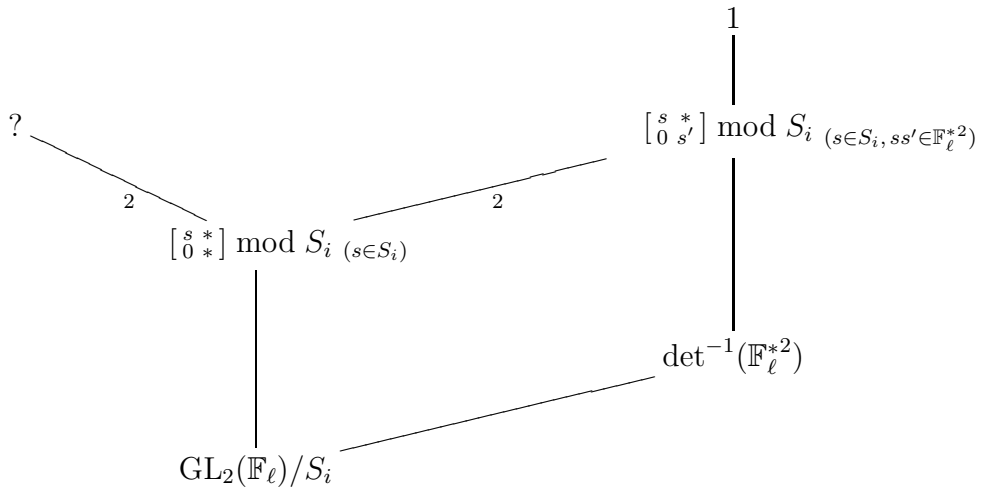
\[ \xymatrix{
&& 1 \\
? & & {\phantom{_{(s \in S_i, \, ss' \in \Flx^2)}}} {\left[ \begin{smallmatrix} s & * \\ 0 & s' \end{smallmatrix} \right] } \bmod S_i \ _{(s \in S_i, \, ss' \in \Flx^2)} \ar@{-}[u] \\
& {\phantom{ _{(s \in S_i)}} } {\left[ \begin{smallmatrix} s & * \\ 0 & * \end{smallmatrix} \right] } \bmod S_i \ _{(s \in S_i)} \ar@{-}[ul]^2 \ar@{-}[ur]_2 & \\
&& \det^{-1}(\Flx^2) \ar@{-}[uu] \\
& \GLFl/S_i \ar@{-}[uu] \ar@{-}[ur] & \\
} \]
and since $\{ \left[ \begin{smallmatrix} s & * \\ 0 & * \end{smallmatrix} \right] \bmod S_i, \, s \in S_i \} \simeq \Fl \rtimes \Flx$ has only one subgroup of index $2$, we expect \ref{A2} to hold for the small factors of $Q_i(x)$ with $j=0$.
\end{rk}

\bigskip
\noindent\framebox{\textbf{Step 2: Central 2-cyclic extensions of $\PGL_2(\Fl)$}}
\bigskip

In what follows, for $n \in \N$ we denote by $C_n$ the cyclic group of order $n$. In order to go on with the proof, we will need to know the classification of the central extensions of $\PGL_2(\Fl)$ by $C_{2^i}$, $i \in \N$. 

It is well-known (cf. for instance \cite[theorem 1.2.5]{CohomNF}) that given a group $G$ and a $G$-module $M$, the extensions of $G$ by $M$ such that the conjugation action of lifts of elements of $G$ on $M$ corresponds to the $G$-module structure on $M$ are classified by the cohomology group $H^2(G,M)$. The class of the cocycle $\beta \colon G \times G \longrightarrow M$ corresponds to the set $M \times G$ endowed with the group law
\[ (m,g)\cdot (m',g') = \big(m+g \cdot m'+\beta(g,g'),g g' \big). \]
In particular, the following result is immediate:

\begin{lem}\label{Ex_lift_ord2}
Consider a (necessarily central) extension
\[ 1 \longrightarrow C_2 \longrightarrow \widetilde G \longrightarrow G \longrightarrow 1 \]
of a group $G$ by $C_2$. Let $\beta \colon G \times G \longrightarrow C_2$ be a cocycle representing the corresponding cohomology class, and let $g \in G$ be an element of $G$ of order $2$. Then the lifts of $g$ in $\widetilde G$ have order $2$ if $\beta(g,g)$ is trivial, but have order $4$ else.
\end{lem}

\bigskip

Furthermore (cf. \cite[theorem 2.1.19]{Karpi}), if the $G$-action on $M$ is trivial, then there is a split exact sequence of Abelian groups
\begin{equation}\label{Eq_H2split}
\xymatrix{ 0 \ar[r] & \Ext^1_\Z(G^{\ab},M) \ar@{^(->}[r]^-{\phi} & H^2(G,M) \ar@<2pt>[r]^-{\psi} & \Hom\big(\widehat{M},H^2(G,\C^*)\big) \ar@<2pt>[l] \ar[r] & 0 \\
} \tag{$\star$}
\end{equation}
where $\Ext^1_\Z(G^{\ab},M)$ classifies the abelian extensions of the abelianised $G^{\ab}$ of $G$ by $M$, $\widehat{M} = \Hom(M,\C^*)$ is the group of complex-valued characters on $M$, $H^2(G,\C^*)$ (with trivial $G$-action on $\C^*$) is the so-called \emph{Schur multiplier} of $G$, and $\psi$ maps the class of the cocycle $\beta \in H^2(G,M)$ to the \emph{transgression map} (not to be confused with a trace)
\[ \begin{array}{cccc}
\Tra_\beta \colon & \widehat{M} & \longrightarrow & H^2(G,\C^*) \\
& \chi & \longmapsto & \chi \circ \beta
\end{array}
\]
associated to the class of $\beta$. Besides, the Schur multiplier $H^2(G,\C^*)$ is trivial if $G$ is cyclic (cf. \cite[proposition 2.1.1.(ii)]{Karpi}), and for each central extension $\widetilde{G}$ of $G$ by $M$, the subgroup $M \cap D\widetilde{G}$ of $\widetilde{G}$ is isomorphic to the image of $\Tra_\beta$, where $\beta \in H^2(G,M)$ is the cohomology class corresponding to $\widetilde G$, and $D\widetilde{G}$ denotes the commutator subgroup of $\widetilde G$ (cf. \cite[proposition 2.1.7]{Karpi}).

Applying this to the group $G=\PGL_2(\Fl)$ and the trivial $G$-module $M = C_{2^i}$ yields the following result (cf. \cite{Quer}):

\pagebreak

\begin{thm}\label{Thm_Quer}
Let $i \in \N$.

\vspace{-1mm}

\begin{enumerate}[label=(\roman*)]
\item $H^2\big(\PGL_2(\Fl),C_{2^i} \big) \simeq C_2 \times C_2$, so that there are four central extensions of $\PGL_2(\Fl)$ by $C_{2^i}$.

\item These extensions are

\vspace{-3mm}

\begin{itemize}
\item the trivial extension $C_{2^i} \times \PGL_2(\Fl)$, corresponding to the trivial cohomology class $\beta_0 \in H^2\big(\PGL_2(\Fl),C_{2^i} \big)$,
\item the group $2^i_{\det} \! \PGL_2(\Fl)$, whose class $\beta_{\det} \in H^2\big(\PGL_2(\Fl),C_{2^i} \big)$ is the inflation of the non-trivial element of
\[ H^2\big(\PGL_2(\Fl)^{\ab},C_{2^i} \big) \simeq C_2 \]
(in other words, $\beta_{\det}(g,g')$ is non-zero if and only if neither $g$ nor $g'$ lie in $\PSL_2(\Fl)$),
\item the group $2^i_- \! \PGL_2(\Fl)$, with class $\beta_- \in H^2\big(\PGL_2(\Fl),C_{2^i} \big)$, defined for $i=1$ as
\[ 2_- \! \PGL_2(\Fl) = \SL_2(\Fl) \sqcup \left[ \begin{smallmatrix} \sqrt{\varepsilon} & 0 \\ 0 & 1/\sqrt{\varepsilon} \end{smallmatrix} \right] \SL_2(\Fl) \subset \SL_2(\F_{\ell^2}) \]
where $\varepsilon$ denotes a generator of $\Flx$, and that for $i \geqslant 2$ corresponds to the image of the cohomology class of $2_- \! \PGL_2(\Fl)$ by the map
\[ H^2\big(\PGL_2(\Fl),C_2 \big) \longrightarrow H^2\big(\PGL_2(\Fl),C_{2^i} \big) \]
induced by the embedding of $C_2$ into $C_{2^i}$,
\item and the group $2^i_+ \! \PGL_2(\Fl)$, whose associated cohomology class $\beta_+$ is the sum in $H^2\big(\PGL_2(\Fl),C_{2^i} \big)$ of $\beta_{\det}$ and of $\beta_-$.
\end{itemize}

\item Identify $C_2$ with $\Z/2\Z$, let $g \in \PGL_2(\Fl)$ be an element of order $2$, and let $\beta_0$, $\beta_{\det}$, $\beta_-$ and $\beta_+$ be normalised cocycles (that is to say $\beta(1,h) = \beta(h,1) = 0$ for all $h \in \PGL_2(\Fl)$) representing the cohomology classes of these four extensions. If $i=1$, then their value at $(g,g)$ does not depend on the choice of these cocycles, and are

\vspace{-3mm}

\begin{itemize}
\item $\beta_0(g,g) = 0 \ \forall g,$
\item $\beta_{\det}(g,g) = \left\{ \begin{array}{ll} 0, & g \in \PSL_2(\Fl), \\ 1, & g \not \in \PSL_2(\Fl), \end{array} \right.$
\item $\beta_-(g,g) = 1 \ \forall g \text{ of order } 2,$
\item $\beta_+(g,g) = \left\{ \begin{array}{ll} 1, & g \in \PSL_2(\Fl), \\ 0, & g \not \in \PSL_2(\Fl). \end{array} \right.$
\end{itemize}

\item For $i \geqslant 2$, the abelianisations of these extensions are

\vspace{-3mm}

\begin{itemize}
\item $\big( C_{2^i} \times \PGL_2(\Fl) \big)^{\ab} \simeq C_{2^i} \times C_2$,
\item $\big( 2^i_{\det} \! \PGL_2(\Fl) \big)^{\ab} \simeq C_{2^{i+1}}$,
\item $\big( 2^i_- \! \PGL_2(\Fl) \big)^{\ab} \simeq C_{2^{i-1}} \times C_2$,
\item $\big( 2^i_+ \! \PGL_2(\Fl) \big)^{\ab} \simeq C_{2^i}$.
\end{itemize}
\end{enumerate}
\end{thm}

\begin{proof}
We shall only give the idea of the proof here, and refer the reader to \cite[proposition 2.4 and lemma 3.2]{Quer}.
\begin{enumerate}[label=(\roman*)]
\item On the one hand, the abelianised of $\PGL_2(\Fl)$ is $\PGL_2(\Fl)/\PSL_2(\Fl) \simeq C_2$, so that
\[ \Ext^1_\Z(\PGL_2(\Fl)^{\ab},C_{2^i}) \simeq \Ext^1_\Z(C_2,C_{2^i}) \simeq C_2. \]
On the other hand, the Schur multiplier $H^2\big(\PGL_2(\Fl),\C^*\big)$ is isomorphic to $C_2$ (cf. \cite[proposition 2.3]{Quer}). The result then follows from the split exact sequence \eqref{Eq_H2split}.

\item Consider again the exact sequence \eqref{Eq_H2split}. Then $\beta_{\det}$ lies in the image of $\phi$ since it is inflated from $\PGL_2(\Fl)^{\ab}$. On the other hand, for $i=1$, $\beta_-$ does not lie in $\Im \phi$, for if it did, then the associated transgression map would be trivial, so that the commutator subgroup of $2_- \! \PGL_2(\Fl)$ would meet the kernel $\pm \left[ \begin{smallmatrix} 1 & 0 \\ 0 & 1 \end{smallmatrix} \right]$ of the extension trivially, which is clearly not the case since $\left[ \begin{smallmatrix} -1 & 0 \\ 0 & -1 \end{smallmatrix} \right]$ is a commutator in $\SL_2(\Fl) \subset 2_- \! \PGL_2(\Fl)$. For $i\geqslant 2$, the commutative diagram
\[ \xymatrix{
1 \ar[r] & C_2 \ar[r] \ar@{^(->}[d] & {2_- \! \PGL_2(\Fl)} \ar[r] \ar@{^(->}[d] & \PGL_2(\Fl) \ar[r] \ar@{^(->}[d] & 1 \\
1 \ar[r] & C_{2^i} \ar[r] & {2^i_- \! \PGL_2(\Fl)} \ar[r] & \PGL_2(\Fl) \ar[r] & 1 \\
} \]
shows that $C_{2^i}$ still intersects the commutator subgroup of $2^i_- \! \PGL_2(\Fl)$ non-trivially, so that $\beta_-$ does not lie in $\Im \phi$ either. The extensions $2^i_{\det} \! \PGL_2(\Fl)$ and $2^i_- \! \PGL_2(\Fl)$ thus represent different non-trivial cohomology classes in $H^2\big(\PGL_2(\Fl),C_{2^i} \big) \simeq C_2 \times C_2$, hence the result.

\item It is a general fact (cf. \cite[lemma 3.1]{Quer} that the image at $(g,g)$ of a normalised cocycle representing an extension of a group $G$ by $C_2$ only depends on the cohomology class of this cocycle in $H^2(G,C_2)$.
\begin{itemize}
\item The case of the trivial extension is obvious since the zero cohomology class is represented by the zero cocycle.
\item The case of $\beta_{\det}$ follows from its very definition.
\item Since it is a subgroup of $\SL_2(\F_{\ell^2})$, the group $2_- \! \PGL_2(\Fl)$ has only one element of order $2$, namely the central element $\left[ \begin{smallmatrix} -1 & 0 \\ 0 & -1 \end{smallmatrix} \right]$. In particular, no element $g \in \PGL_2(\Fl)$ of order $2$ remains of order $2$ when lifted to $2_- \! \PGL_2(\Fl)$, and the result follows from lemma \ref{Ex_lift_ord2}.
\item The case of $\beta_+$ follows since we may take $\beta_+ = \beta_{\det} + \beta_-$.
\end{itemize} 

\item Again, the case of the trivial extension is clear. In the other cases, the result follows from the fact that the intersection of $C_{2^i}$ with the commutator subgroup of the extension is isomorphic to the image of the transgression map
\vspace{-2mm}
\[ \Tra_\beta \colon \widehat{C_{2^i}} \longrightarrow H^2\big( \PGL_2(\Fl),\C^*\big) \simeq C_2,\]
which is trivial in the case of $\beta_{\det}$ and non-trivial in the case of $\beta_-$ and $\beta_+$.
\vspace{-5mm}
\end{enumerate}
\end{proof}

We shall now use this classification to prove by elimination that $\Gal(L_i/\Q)$ is isomorphic to $\GL_2(\Fl)/S_i$ for all $i$.

\begin{rk}
The group $\GLFl/S_i$ must be one of the cases presented in theorem \ref{Thm_Quer}, but at this point it is not clear at all which one. We will eventually determine this, cf. remark \ref{Rk_WhoIsGL_S} below.
\end{rk}

\bigskip
\noindent\framebox{\textbf{Step 3: The case of $L_1/L_0$}}
\bigskip

We first deal with the first extension $L_1/L_0$ in the quadratic tower $L_r / \cdots / L_0$. The Galois group $\Gal(L_1/\Q)$ is a (necessarily central) extension of $\Gal(L_0/\Q) \simeq \PGL_2(\Fl)$ by $C_2$. 
\begin{pro}\label{Gal_L1}
$\Gal(L_1/\Q)$ is isomorphic to $\GL_2(\Fl) / S_1$ as an extension of $\PGL_2(\Fl)$.
\end{pro}

\begin{proof}
Let $\beta$ be a normalised cocycle representing the cohomology class corresponding to the extension $\Gal(L_1/\Q)$ of $\PGL_2(\Fl)$. According to theorem \ref{Thm_Quer}(ii), $\Gal(L_1/\Q)$ is isomorphic either to $C_2 \times \PGL_2(\Fl)$, $2_{\det} \! \PGL_2(\Fl)$, $2_- \! \PGL_2(\Fl)$ or $2_+ \! \PGL_2(\Fl)$, and $\beta$ is correspondingly cohomologous to $\beta_0$, $\beta_{\det}$, $\beta_-$ or $\beta_+$.

If $\Gal(L_1/\Q)$ were the trivial extension $C_2 \times \PGL_2(\Fl)$, then $L_1$ would have a subextension $L_1^{\ab}$ with Galois group isomorphic to
\[ \big( C_2 \times \PGL_2(\Fl) \big)^{\ab} \simeq C_2 \times C_2, \]
and hence three distinct quadratic subfields, which contradicts \ref{A1}.

Let now $\tau_1 \in \Gal(L_1/\Q)$ be the complex conjugation relative to some embedding of $L_1$ into $\C$. It induces an element $\tau_0 \in \Gal(L_0/\Q)$, which is not the identity since its image by $\rho_{f,\l}^{\text{proj}}$ is conjugate to $g = \left[ \begin{smallmatrix} 1 & 0 \\ 0 & -1 \end{smallmatrix} \right] \in \PGL_2(\Fl)$. In particular, $\tau_1$ is not trivial either, so it has order $2$. Therefore $\tau_0$ has a lift to $\Gal(L_1/\Q)$ of order $2$, so that $\beta(\tau_0,\tau_0)$ is trivial by lemma \ref{Ex_lift_ord2}. Theorem \ref{Thm_Quer}(iii) then only leaves one possibility: if $\ell \equiv 1 \bmod 4$, then $g \in \PSL_2(\Fl)$, so that $\beta$ cannot be cohomologous to $\beta_-$ nor to $\beta_+$ and so $\Gal(L_1/\Q)$ must be isomorphic to $2_{\det} \! \PGL_2(\Fl)$, whereas if $\ell \equiv -1 \bmod 4$, then $g \not \in \PSL_2(\Fl)$, so that $\beta$ cannot be cohomologous to $\beta_-$ nor to $\beta_{\det}$ and so $\Gal(L_1/\Q)$ must be isomorphic to $2_{+} \! \PGL_2(\Fl)$.

Besides, $\saint L_1$ is a quadratic extension of $\saint L_0$ and has only one quadratic subfield since its Galois group is isomorphic to $\GLFl/S_1$, so that the same reasoning applies and shows that $\Gal(\saint L_1/\Q)$ is isomorphic to $2_{\det} \! \PGL_2(\Fl)$ if $\ell \equiv 1 \bmod 4$ and to $2_{+} \! \PGL_2(\Fl)$ if $\ell \equiv 1 \bmod 4$. Either way, we have
\[ \Gal(L_1/\Q) \simeq \Gal(\saint L_1/\Q) \simeq \GL_2(\Fl) / S_1. \]
\end{proof}

\bigskip
\noindent\framebox{\textbf{Step 4: $\Gal(L_i/\Q)$ is an extension of $\PGL_2(\Fl)$ by $C_{2^i}$}}
\bigskip

If $\ell \equiv -1 \bmod 4$, then $r=1$, so that the proof that $\Gal(L_r/\Q) \simeq \GL_2(\Fl)/S_r$ is over. We therefore assume that $\ell \equiv 1 \bmod 4$ henceforth until we finish proving part (i) of theorem \ref{thm_preuve}. We shall first prove by induction on $i$ that $\Gal(L_i/\Q)$ is an extension of $\PGL_2(\Fl)$ by $\Flx/S_i \simeq C_{2^i}$, then that this extension is central, and finally that it is isomorphic to $\GLFl/S_i$. Note that we have just proved above that it is so for $i = 1$.

\bigskip
 
We first prove that $\Gal(L_i/\Q)$ is an extension of $\PGL_2(\Fl)$ by $C_{2^i}$. Let $1~\leqslant~i<~r$. By induction hypothesis, we have a commutative diagram
\[ \xymatrix @!0 @R=3cm @C=3cm {
& 1 \ar[d] & 1 \ar[d] & 1 \ar[d] & \\
1 \ar[r] & C_2 \ar@{=}[d] \ar[r]^-{j} & q^{-1}(C_{2^i}) \ar[d]^{\iota} \ar[r]^-{q} & C_{2^i} \ar[d]^{\iota} \ar[r] & 1 \\
1 \ar[r] & C_2 \ar[r]^-{j} & \Gal(L_{i+1}/\Q) \ar[r]^q \ar[rd]_{\pi \circ q} & \Gal(L_{i}/\Q) \ar[r] \ar[d]^\pi & 1 \\
& & & \PGL_2(\Fl) \ar[rd] \ar[d] & \\
& & & 1 & 1 \\
} \]
whose middle row and right column are exact. A diagram chase then reveals that the top row and the diagonal short sequence
\[ 1 \longrightarrow q^{-1}(C_{2^i}) \overset{\iota}{\longrightarrow} \Gal(L_{i+1}/\Q) \overset{\pi \circ q}{\longrightarrow} \PGL_2(\Fl) \longrightarrow 1 \]
are exact, so that $\Gal(L_{i+1}/\Q)$ is an extension of $\PGL_2(\Fl)$ by $q^{-1}(C_{2^i})$, which itself is an extension of $C_{2^i}$ by $C_2$, which is necessarily central since $\Aut(C_2)$ is trivial.

We have $H^2(C_{2^i},\C^*) = \{0\}$ because $C_{2^i}$ is cyclic, so the extensions of $C_{2^i}$ by $C_2$ are all Abelian by the exact sequence \eqref{Eq_H2split}, so that $q^{-1}(C_{2^i}) = \Gal(L_{i+1}/L_0)$ is isomorphic either to $C_{2^{i+1}}$ or to $C_{2^i} \times C_2$. We shall now prove that the latter is impossible.

Since $\ell \equiv 1 \bmod 4$, the group $S_1^2 = \Flx^4$ is a strict subgroup of $S_1 = \Flx^2$. The determinant induces a surjective morphism
\[ \xymatrix{\Gal(L_1/\Q) \ar[r]^{\sim} & \GL_2(\Fl) / S_1 \ar@{->>}[r]^-{\det} & \Flx / S_1^2 = \Flx / \Fl^{*4} \simeq C_4, } \]
so that $L_1$ has a $C_4$-subfield, which can only be field $\kappa_2 \subset \Q(\mu_\ell)$ according to \ref{A1}.

\pagebreak

Besides, $\kappa_2$ cannot be contained in $L_0$ because $\PGL_2(\Fl)^{\ab} \simeq C_2$, and since $\kappa_2$ is a quadratic extension of $\kappa_1 = \Q(\sqrt{\ell^*}) \subset L_0$ and $L_1$ is a quadratic extension of $L_0$, we have $L_1 = \kappa_2 L_0$:
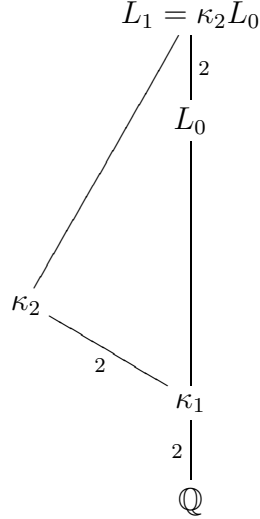
\[ \xymatrix{
& L_1=\kappa_2 L_0 \\
& L_0 \ar@{-}[u]_2  \\
 \\
\kappa_2 \ar@{-}[uuur] && \\
& \kappa_1 \ar@{-}[uuu] \ar@{-}[ul]^2  \\
& \Q \ar@{-}[u]^2  \\
} \] 
Now if $\Gal(L_{i+1}/L_0)$ were isomorphic to $C_{2^i} \times C_2$, then, letting $E$ be the subfield of $L_{i+1}$ fixed by ${C_{2^i} \times \{ 1 \}}$, we would have the extension tower
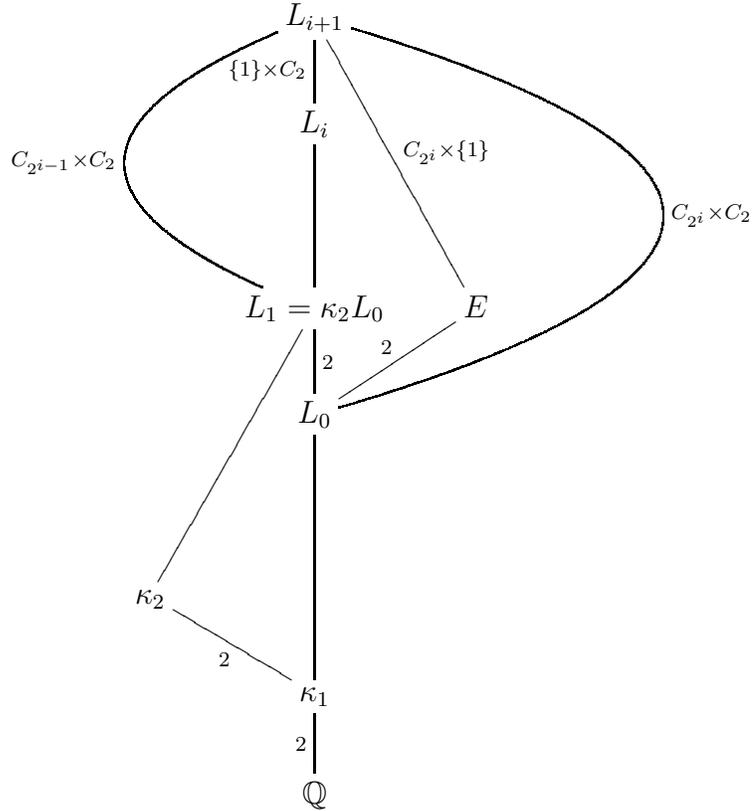
\[ \xymatrix{
& L_{i+1} & \\
& L_{i} \ar@{-}[u]^{\{1\} \times C_2} & \\
&& \\
& L_1=\kappa_2 L_0 \ar@{-}[uu] \ar@/^6pc/@{-}[uuu]^{C_{2^{i-1}} \times C_2} & E \ar@{-}[uuul]_{C_{2^i} \times \{1 \}} \\
& L_0 \ar@{-}[ur]^2 \ar@{-}[u]_2 \ar@/_11pc/@{-}[uuuu]_{C_{2^i} \times C_2} & \\
&& \\
\kappa_2 \ar@{-}[uuur] && \\
& \kappa_1 \ar@{-}[uuu] \ar@{-}[ul]^2 & \\
& \Q \ar@{-}[u]^2 & \\
} \] 
where $C_{2^{i-1}}$ denotes the subgroup of $C_{2^i}$ of index $2$. The extensions $E/L_0$ and $L_1/L_0$ are both quadratic subextensions of $L_{i+1}/L_0$, but they are distinct since they correspond respectively to the distinct subgroups $C_{2^i} \times \{ 1 \}$ and $C_{2^{i-1}} \times C_2$ of $\Gal(L_{i+1}/L_0) = C_{2^i} \times C_2$. On the other hand, the field $E$ is contained in $L_{i+1}$ and thus has only one quadratic subfield according to \ref{A1}, so that the same reasoning as in step 3 above shows that $\Gal(E/\Q)$ is isomorphic to $\GL_2(\Fl)/S_1$. But then $E$ has a $C_4$-subfield, which can only be $\kappa_2$, and so $E \supseteq \kappa_2 L_0 = L_1$, hence $E = L_1$ since they are both quadratic extensions of $L_0$, a contradiction.

This shows that $\Gal(L_{i+1}/L_0)$ cannot be isomorphic to $C_{2^i} \times C_2$, so must be isomorphic to $C_{2^{i+1}}$. It follows that $\Gal(L_{i+1}/\Q)$ is an extension of $\PGL_2(\Fl)$ by $\Gal(L_{i+1}/L_0) \simeq C_{2^{i+1}}$, and the induction is complete.

\bigskip
\noindent\framebox{\textbf{Step 5: $\Gal(L_i/\Q)$ is a \emph{central} extension of $\PGL_2(\Fl)$}}
\bigskip

We shall now prove by induction on $i$ that the extension
\[ 1 \longrightarrow C_{2^i} \longrightarrow \Gal(L_i/\Q) \longrightarrow \PGL_2(\Fl) \longrightarrow 1 \]
is central. Note that it is so for $i=1$ since $\Aut(C_2)$ is trivial.

Let $i \geqslant 2$, and assume on the contrary that this extension is not central. Since $\Aut(C_{2^i}) \simeq C_{2^{i-1}}$ is Abelian, the morphism $\PGL_2(\Fl) \longrightarrow \Aut(C_{2^i})$ expressing the conjugation action of $\PGL_2(\Fl)$ on $C_{2^i}$ factors through $\PGL_2(\Fl)^{\ab} = \PGL_2(\Fl)/\PSL_2(\Fl) \simeq C_2$, so that $\PSL_2(\Fl)$ acts trivially whereas there exists an involution $\phi$ of $C_{2^i}$ such that $g x g^{-1} = \phi(x)$ for all $g \in \PGL_2(\Fl) - \PSL_2(\Fl)$ and $x \in C_{2^i}$. If we identify $C_{2^i}$ with $\Z/2^i\Z$, then by induction hypothesis this involution induces the identity on $\Z/2^{i-1}\Z$, so it must be $x \mapsto (1+2^{i-1})x$.

There is thus only one possible non-trivial conjugation action of $\PGL_2(\Fl)$. In order to compute $H^2\big(\PGL_2(\Fl),C_{2^i}\big)$ for this non-trivial action, we use the inflation-restriction exact sequence (cf. \cite[proposition VII.6.5]{Serre_LocalFields})
\begin{equation*}
0 \longrightarrow H^2(C_2,C_{2^i}) \overset{\Inf}{\longrightarrow} H^2\big(\PGL_2(\Fl),C_{2^i}\big) \overset{\Res}{\longrightarrow} H^2\big(\PSL_2(\Fl),C_{2^i}\big).
\label{Eq_InflRes} \tag{$\dagger$}
\end{equation*}
This is legitimate since, as $\PSL_2(\Fl)$ acts trivially, we have
\[ H^1\big(\PSL_2(\Fl),C_{2^i}\big) = \Hom\big(\PSL_2(\Fl),C_{2^i}\big) = 0 \]
since $\PSL_2(\Fl)$ is simple.

On the one hand, since $C_2=\{1,\varepsilon\}$ is cyclic, the groups $H^q(C_2,M)$ are the cohomology groups of the complex
\[ 0 \longrightarrow M \overset{\varepsilon-1}{\longrightarrow} M \overset{\varepsilon+1}{\longrightarrow} M \overset{\varepsilon-1}{\longrightarrow} M \overset{\varepsilon+1}{\longrightarrow} \cdots \]
for any $C_2$-module $M$ (cf. \cite[chapter XX exercise 16]{LangAlgebra}). In particular,
\[ H^2(C_2,C_{2^i}) = \frac{\ker(\varepsilon-1)}{\Im(\varepsilon+1)} = \frac{(\Z/2^i\Z)[2^{i-1}]}{(2+2^{i-1})(\Z/2^i\Z)} \simeq \left\{ \begin{array}{ll} C_2, & i=2, \\ 0, & i \geqslant 3. \end{array} \right. \]

On the other hand, as $\PSL_2(\Fl)$ acts trivially, the group $H^2\big(\PSL_2(\Fl),C_{2^i}\big)$ can be computed by using the split exact sequence \eqref{Eq_H2split}. As $\PSL_2(\Fl)^{\ab}=\{1\}$ since $\PSL_2(\Fl)$ is simple, and as the Schur multiplier is
\[ H^2\big(\PSL_2(\Fl),\C^*\big) \simeq C_2 \]
(Steinberg, cf. \cite[theorem 7.1.1.(ii)]{Karpi}), it follows that 
\[ H^2\big(\PSL_2(\Fl),C_{2^i}\big) \simeq C_2. \]
Let $2^i \! \PSL_2(\Fl)$ denote the non-trivial extension. One has
\[ 2 \! \PSL_2(\Fl) \simeq \SL_2(\Fl), \]
and the non-trivial element of $H^2\big(\PSL_2(\Fl),C_{2^i}\big)$ is the image of the non-trivial element $\gamma_{\SL_2} \in H^2\big(\PSL_2(\Fl),C_2\big)$ corresponding to $\SL_2(\Fl)$ by the map
\[ H^2\big(\PSL_2(\Fl),C_2\big) \longrightarrow H^2\big(\PSL_2(\Fl),C_{2^i}\big) \]
induced by the embedding of $C_2$ into $C_{2^i}$.

Consider the inflation-restriction exact sequence \eqref{Eq_InflRes}, and let
\[ \beta \in H^2\big(\PGL_2(\Fl),C_{2^i}\big) \]
be the cohomology class corresponding to the extension
\[ 1 \longrightarrow C_{2^i} \longrightarrow \Gal(L_i/\Q) \longrightarrow \PGL_2(\Fl) \longrightarrow 1. \]
If $\gamma = \Res \beta \in H^2\big(\PSL_2(\Fl),C_{2^i}\big)$ were trivial, then $\beta = \Inf \alpha$ would be the inflation of some $\alpha \in H^2\big(C_2,C_{2^i}\big)$, so that $\Gal(L_i/\Q)$ would be isomorphic to the fibred product $G_\alpha \underset{C_2}{\times} \PGL_2(\Fl)$, where $G_\alpha$ is the group extension
\[ 1 \longrightarrow C_{2^i} \longrightarrow G_\alpha \longrightarrow C_2 \longrightarrow 1 \]
corresponding to $\alpha$. Actually, if $i\geqslant 3$, then $\beta = \Inf \alpha$ would be trivial since $H^2\big(C_2,C_{2^i}\big)=0$, so that $\Gal(L_i/\Q)$ would be isomorphic to the semi-direct product
\[ C_{2^i} \rtimes \PGL_2(\Fl), \]
whereas if $i=2$, then $H^2\big(C_2,C_{2^i}\big)\simeq C_2$, so that $\Gal(L_2/\Q)$ would be isomorphic either to $C_4 \rtimes \PGL_2(\Fl)$ or to $Q_8 \underset{C_2}{\times} \PGL_2(\Fl)$, where $Q_8$, the quaternionic group $\{ \pm 1, \pm i, \pm j, \pm k \}$, is the extension
\[ 1 \longrightarrow C_4 \longrightarrow Q_8 \longrightarrow C_2 \longrightarrow 1 \]
corresponding to the non-trivial element of $H^2(C_2,C_4)$. However, the abelianisations
\[ \Big(C_{2^i} \rtimes \PGL_2(\Fl) \Big)^{\ab} \simeq C_{2^{i-1}} \times C_2 \]
and
\[ \big( Q_8 \underset{C_2}{\times} \PGL_2(\Fl) \big)^{\ab} \simeq C_2 \times C_2 \]
contradict \ref{A1}.

It follows that $\gamma = \Res \beta \in H^2\big(\PSL_2(\Fl),C_{2^i}\big)$ cannot be trivial, so it must be $\gamma_{\SL_2} \in H^2\big(\PSL_2(\Fl),C_2\big)$ followed by the embedding of $C_2$ into $C_{2^i}$. Let $g = \left[ \begin{smallmatrix} 1 & 0 \\ 0 & -1 \end{smallmatrix} \right] \in \PGL_2(\Fl)$. As $\ell \equiv 1 \bmod 4$, $g$ lies in $\PSL_2(\Fl)$, and since the only element of order $2$ of $\SL_2(\Fl)$ is $\left[ \begin{smallmatrix} -1 & 0 \\ 0 & -1 \end{smallmatrix} \right]$, $g$ cannot be lifted to an element of order $2$ of $\SL_2(\Fl)$, so that $\gamma_{\SL_2}(g,g) \neq 0$ by lemma \ref{Ex_lift_ord2}. On the other hand, since $g$ is the image of the complex conjugation (with respect to some embedding of $L_0$ into $\C$) by the projective Galois representation $\rho^{\text{proj}}$, it must lift to an element of order $2$ of $\Gal(L_i/\Q)$, which is contradictory: in the extension $\Gal(L_i/\Q)$, seen as the set $\Z/{2^i}\Z \times \PGL_2(\Fl)$ endowed with the group law
\[ (x_1,g_1) \cdot (x_2,g_2) = \big(x_1 + g_1 \cdot x_2 + \beta(g_1,g_2), g_1 g_2 \big), \]
we compute that
\[ (x,g)\cdot(x,g) = \big( x+g \cdot x + \beta(g,g), g^2 \big) = \big( \beta(g,g), 1 \big) \]
as $g \in \PSL_2(\Fl)$ acts trivially, so $\beta(g,g)$ must be zero, but $\beta(g,g) = \gamma_{\SL_2}(g,g) \neq 0$ since $g \in \PSL_2(\Fl)$.

It is therefore impossible that the extension
\[ 1 \longrightarrow C_{2^i} \longrightarrow \Gal(L_i/\Q) \longrightarrow \PGL_2(\Fl) \longrightarrow 1 \]
be not central, which completes the induction.

\bigskip
\noindent\framebox{\textbf{Step 6: $\Gal(L_i/\Q) \simeq \GL_2(\Fl)/S_i$}}
\bigskip

We may now apply again theorem \ref{Thm_Quer} to $\Gal(L_r/\Q)$. Part (iv) of this theorem combined with \ref{A1} means that $\Gal(L_r/\Q)$ cannot be isomorphic to $C_{2^r} \times \PGL_2(\Fl)$ nor to $2^r_{-} \! \PGL_2(\Fl)$. It cannot be isomorphic to $2^r_{\det} \! \PGL_2(\Fl)$ either, else $L_r$ would have a $C_{2^{r+1}}$-subfield by part (iv) of theorem \ref{Thm_Quer}, which would be contained in the cyclotomic extension $\Q(\mu_\ell)$ according to \ref{A1}, but this would contradict the definition of $r$.
Therefore, $\Gal(L_r/\Q)$ must therefore be isomorphic to $2^r_{+} \! \PGL_2(\Fl)$.

Besides, the same reasoning applies to $\saint L_r$, whose Galois group is isomorphic $\GL_2(\Fl)/S_r$ since $\det \saint \rho$ is by assumption an odd power of the mod $\ell$ cyclotomic character. Therefore, we have
\[ \Gal(L_r/\Q) \simeq 2^r_{+} \! \PGL_2(\Fl) \simeq \Gal(\saint L_r / \Q) \simeq \GL_2(\Fl)/S_r, \]
and the proof of part (i) of theorem \ref{thm_preuve} is now complete.

\begin{rk}\label{Rk_WhoIsGL_S}
We can now go back down the quadratic tower $L_r/ \cdots / L_0$ and see that $\Gal(L_i/\Q) \simeq \GL_2(\Fl)/S_i$ for all $i$. Besides, it is easy to see that the abelianisation of $\GL_2(\Fl)/S_i$ is
\[ \det \colon \GL_2(\Fl)/S_i \longrightarrow \Flx/S_i^2, \]
and since $S_i^2 = S_{i+1} \subsetneq S_i$ for $i < r$ whereas $S_r^2 = S_r$ as $-1 \not \in S_r$, theorem~\ref{Thm_Quer} part~(iv) leads to the unified formula
\[ \Gal(L_i/\Q) \simeq \GL_2(\Fl)/S_i \simeq \left\{ \begin{array}{ll} \PGL_2(\Fl), & i=0, \\ 2^i_{\det} \! \PGL_2(\Fl), & 0 < i < r, \\ 2^r_+ \! \PGL_2(\Fl), & i=r, \end{array} \right. \]
which is valid for $\ell \equiv 1 \bmod 4$ as well as $\ell \equiv -1 \bmod 4$. This allows us to identify for each $i$ the extension $\GL_2(\Fl)/S_i$ of $\PGL_2(\Fl)$ amongst the ones listed in part (ii) of theorem \ref{Thm_Quer}.
\end{rk}

\subsubsection{Certification of the Galois action}\label{Certif_Galois_action}

At this point, we have proved that $\Gal(L_i/\Q)$ is abstractly isomorphic to $\GLFl/S_i$ for each $0 \leqslant  i \leqslant r$, but only for $i=0$ do we know that it is permutation-isomorphic to $\GLFl/S_i$ acting naturally on $V_i = V/S_i$. For each $i>0$, we will now determine an isomorphism between $\Gal(L_i/\Q)$ and $\GLFl/S_i$ and a bijection $\theta_i \colon Z_i \overset{\sim}{\longrightarrow} V_i$ which make the Galois action on $Z_i$ permutation-isomorphic to the natural action of $\GLFl/S_i$  on $V_i$ in a compatible way as $i$ varies. This data can then be used to compute the Dokchitsers' resolvents $\Gamma_C(x)$, and thus to compute trace of Frobenius elements, in a certified way. 
 
Let us first fix an isomorphism $\varphi_r$ from the $\Gal(L_r/\Q)$ to $\GLFl/S_r$. Since the Galois groups $\Gal(L_i/\Q)$ are isomorphic to $\GLFl/S_i$ as extensions of $\PGL_2(\Fl)$ in a compatible way, $\varphi_r$ induces a system of isomorphisms
\[ \big(\varphi_i \colon \Gal(L_i/\Q) \simeq \GLFl/S_i \big)_{0\leqslant i \leqslant r} \]
such that the diagram
\[ \xymatrix{
\Gal(L_r/\Q) \ar@{->>}[r] \ar[d]_{\varphi_r}^{\wr} & \cdots \ar@{->>}[r] & \Gal(L_{i+1}/\Q) \ar[d]_{\varphi_{i+1}}^{\wr} \ar@{->>}[r] & \Gal(L_i/\Q) \ar[d]_{\varphi_i}^{\wr} \ar@{->>}[r] & \cdots \ar@{->>}[r] & \Gal(L_0/\Q) \ar[d]_{\varphi_0}^{\wr} \\
\GLFl/S_r \ar@{->>}[r] & \cdots \ar@{->>}[r] & \GLFl/S_{i+1} \ar@{->>}[r] & \GLFl/S_i \ar@{->>}[r] & \cdots \ar@{->>}[r] & \PGL_2(\Fl) \\ 
} \]
commutes. We choose $\varphi_r$ such that the induced isomorphism
\[ \varphi_0 \colon \Gal(L_0/\Q) \simeq \PGL_2(\Fl) \]
agrees with the one we determined with the help of \cite{Magma} in section \ref{GalProj}, and we will use the isomorphisms $\varphi_i$ to identify $\Gal(L_i/\Q)$ with $\GLFl/S_i$ from now on.

Since, by section \ref{GalProj}, the action of $\Gal(L_0/\Q)$ on $Z_0$ is equivalent to the natural action of $\PGL_2(\Fl)$ on $\P^1(\Fl)$, we know that the stabiliser of a root of $F_0(x)$ is conjugate to a group of upper triangular matrices in $\PGL_2(\Fl)$. Therefore, the stabiliser of a root of $F_1(x)$ is a subgroup of index $2$ of the subgroup of upper triangular matrices in $\GLFl / S_1$.

\begin{lem}\label{Lem_sub2Borel}
Let $B$ be a subgroup of $\GLFl$ of the form
\[ B = \left\{ \left[ \begin{smallmatrix} s & x \\ 0 & s' \end{smallmatrix} \right] \ \big\vert \ s \in S, s' \in S', x \in \Fl  \right\}, \]
where $S, S' \leqslant \Flx$ are subgroups of the multiplicative group of $\Fl$. If neither $S$ nor $S'$ is reduced to $\{1\}$, then $B$ has exactly $3$ subgroups of index $2$, namely
\begin{eqnarray*}
\left\{ \left[ \begin{smallmatrix} s & x \\ 0 & s' \end{smallmatrix} \right] \ \big\vert \ s \in S^2 \right\}, \\
\left\{ \left[ \begin{smallmatrix} s & x \\ 0 & s' \end{smallmatrix} \right] \ \big\vert \ s' \in S'^2 \right\}, \\
\text{and } \left\{ \left[ \begin{smallmatrix} s & x \\ 0 & s' \end{smallmatrix} \right] \ \big\vert \ s \in S^2 \Leftrightarrow s' \in S'^2 \right\}, \\
\end{eqnarray*}
where we write $S^2$ for $\{s^2, s \in S\}$, and similarly for $S'^2$.
\end{lem}

\begin{proof}
Since a subgroup of index $2$ is always normal, such a subgroup is the kernel of a non-trivial morphism from $B$ to $C_2$. As the latter group is Abelian, such a morphism factors through the abelianisation of $B$. Let $s \in S$, $s \neq 1$. The identity $ghg^{-1}h^{-1}=\left[ \begin{smallmatrix} 1 & 1-s \\ 0 & 1 \end{smallmatrix} \right]$ where $g = \left[ \begin{smallmatrix} 1 & 1 \\ 0 & 1 \end{smallmatrix} \right]$, $h = \left[ \begin{smallmatrix} s & 0 \\ 0 & 1 \end{smallmatrix} \right] \in B$ shows that $\left[ \begin{smallmatrix} 1 & 1 \\ 0 & 1 \end{smallmatrix} \right]$ is a commutator in $B$, so that the abelinanisation of $B$ is 
\[ \begin{array}{ccc}
B & \longrightarrow & S \times S' \\ \left[ \begin{smallmatrix} s & x \\ 0 & s' \end{smallmatrix} \right] & \longmapsto & (s,s').
\end{array} \]
Therefore, we have canonically
\[ \Hom( B, C_2) \simeq \Hom( S \times S', C_2) \simeq \Hom( S, C_2) \times \Hom( S', C_2). \]
Since $S$ and $S'$ are cyclic because $\Flx$ is, the result follows.
\end{proof}

According to this lemma, the stabiliser of a root of $F_1(x)$ in $\Gal(L_1/\Q)$ could be either 
\begin{eqnarray*}
H_+ = \left\{ \left[ \begin{smallmatrix} s & x \\ 0 & s' \end{smallmatrix} \right] \ \big\vert \ s \in \Flx^2, s' \in \Flx, x \in \Fl \right\} / S_1, \\
H_- = \left\{ \left[ \begin{smallmatrix} s & x \\ 0 & s' \end{smallmatrix} \right] \ \big\vert \ s \in \Flx, s' \in \Flx^2, x \in \Fl \right\} / S_1,
\\ \text{or }H_0 = \left\{ \left[ \begin{smallmatrix} s & x \\ 0 & s' \end{smallmatrix} \right] \ \big\vert \ s, s' \in \Flx, x \in \Fl, ss' \in \Flx^2 \right\} / S_1.
\end{eqnarray*}

However, the nontrivial element $\left[ \begin{smallmatrix} \varepsilon & 0 \\ 0 & \varepsilon \end{smallmatrix} \right] \in \GLFl/S_1$, where $\varepsilon \in \Flx / \Flx^2$, is central and lies in $H_0$, so it lies in the intersection of the conjugates of $H_0$, so that the action of $\GLFl/S_1$ on its $H_0$-cosets is not faithful. Therefore, the stabiliser of a root of $F_1(x)$ must be conjugate either to $H_+$ or to $H_-$.

Consider now the compatible collection of involutory automorphisms
\[ \begin{array}{rcl}
\Psi_i \colon \GLFl/S_i & \longrightarrow & \GLFl/S_i \\ A & \longmapsto & \frac1{\det A} A.
\end{array} \]
Since $\Psi_0$ is the identity on $\PGL_2(\Fl)$, we may replace the isomorphisms $\varphi_i$ with $\Psi_i \circ \varphi_i$ without breaking the compatibility with the identification of $\Gal(L_0/\Q)$ with $\PGL_2(\Fl)$ made in section \ref{GalProj}, and since $\Psi_1$ swaps $H_+$ and $H_-$, we may assume without loss of generality that the stabiliser of a root of $F_1(x)$ is conjugate to $H_+$.

An induction on $i$ then reveals that the stabiliser in $\Gal(L_i/\Q)$ of a root of $F_i(x)$ is conjugate to
\[ \left\{ \left[ \begin{smallmatrix} s & x \\ 0 & y \end{smallmatrix} \right] \ \big\vert \ s \in S_i, y \in \Flx, x \in \Fl \right\} / S_i. \]
Indeed, at each step of the induction, lemma \ref{Lem_sub2Borel} gives us 3 possibilities, but only one of them yields a faithful action of $\GLFl/S_i$ on its cosets, for the same reason as above.

\bigskip

As a consequence, we now know that for each $i$ there exists a bijection
\[ \theta_i \colon Z_i \overset{\sim}{\longrightarrow} V_i \]
which makes the Galois action on $Z_i$ equivalent to the natural action of $\GLFl/S_i$ on $V_i$, so we have proved part (ii) of theorem \ref{thm_preuve}. However, we must make the indexation $\theta_r$ of $Z_r$ by $V_r$ explicit, so as to be able to proceed with the computation of the Dokchitsers' resolvents $\Gamma_C(x)$. We do so as follows.

\subsubsection{Recovering the indexation of the $p$-adic roots}

Recall that we have fixed a large prime $p \in \N$ such that $F_r(x)$ mod $p$ is irreducible. Consider the field $\overline K_{r} = \F_p[t] / F_r(t)$. The $t^{p^j}$, $0 \leqslant j < 2^r(\ell+1)$, are the roots of $F_r$ in $\overline K_{r}$, and so by the hypothesis we have made on the relation between the roots of $F_i(x)$ and the ones of $F_{i+1}(x)$, all the polynomials $F_i(x)$ are squarefree and split completely over $\overline K_{r}$. Let $\overline Z_{i}$ be the set of the roots of $F_i(x)$ in $\overline K_{r}$, so that we have\footnote{Although we certainly have such projections maps in characteristic zero, it might happen that these maps are not well-defined anymore in characteristic $p$. However, as $p$ is large, this problem should not occur for us.} 2-to-1 projection maps $\xymatrix{ \overline \varpi_{i} \colon \overline Z_{i+1} \ar@{->>}[r] & \overline Z_{i}}$ such that for all $z \in \overline Z_{i+1}$, there exists a unique $z' \in \overline Z_{i+1}$ such that $z+z' = \overline \varpi_{i}(z) \in \overline Z_{i}$.

In section \ref{GalProj}, \cite{Magma} computed for us the Galois group $\Gal(L_0/\Q)$ as a permutation group on the roots of $F_0(x)$ in some extension $M$ of $\F_p$, which unfortunately is not isomorphic\footnote{Indeed, unlike $\overline K_r$,  $M$ is an extension of $\F_p$ of degree $\ell+1 = \deg F_0(x)$. To make things worse, curiously Magma does not construct $M$ as $\F_p[t]/F_0(t)$ but as $\F_p[t]/G(t)$ instead, where $G(t)$ is a sparse polynomial of degree $\ell+1$ which it cooks up.} to $\overline K_r$. Magma also gave us an indexation $(m_P)_{P \in \P^1(\Fl)}$ of these roots, and we would like to transfer this indexation to $\overline Z_{0} \subset \overline K_r$ while keeping compatibility with the action of $\Gal(L_0/\Q)= \PGLFl$. We do so by computing mod $p$ the factors
\[ R_{4,P}(x) = \prod_{\substack{P_1,P_2,P_3,P_4 \in \P^1(\Fl) \\ \text{pairwise distinct} \\ [P_1,P_2,P_3,P_4]=P}} \left( x - \sum_{i=1}^{4} \lambda_i m_{P_i} \right) \in \F_p[x] \]
of the resolvent $R_4(x)$ from section \ref{GalProj} for each $P \in \P^1(\Fl) - \{ \infty, 0, 1\}$, where $[\cdot,\cdot,\cdot,\cdot]$ denotes the cross-ratio and the $(\lambda_i)_{1\leqslant i \leqslant 4}$ are fixed distinct integers chosen so that these polynomials are pairwise coprime mod $p$. Although we did mention that the resolvent $R_4(x)$ is horribly expensive to compute, computing these factors is much easier, for three reasons : they are merely factors and so their degree is much smaller, we compute them mod $p$ so the size of their coefficients is no longer a problem, and now we know that $\Gal(L_0/\Q) = \PGLFl$, it is rigorous to compute them by expanding the product that defines them instead of using resultants.

Then, since the action of $\PGLFl$ on $\P^1(\Fl)$ is $3$-transitive, we may index $3$ distinct arbitrarily chosen points $z_\infty$, $z_0$ and $z_1$ of $\overline Z_0$ respectively by $\infty$, $0$ and $1$, after what we index each remaining point $z \in \overline Z_{0}$ by the unique $P \in \P^1(\Fl)$ such that
\[ R_P(\lambda_1 z_\infty + \lambda_2 z_0 + \lambda_3 z_1 + \lambda_4 z) = 0. \]
Next, by looking at how the Frobenius of $\overline K_{r}$ permutes $\overline Z_{0}$, we may deduce which element $\overline \Phi \in \PGLFl$ it corresponds to.

Let now $z = z^{(r)} \in \overline Z_{r}$ be a fixed root of $F_r(x)$ in $\overline K_{r}$. By finding which other point of $\overline Z_{r}$ must be added to it to get a root $z^{(r-1)}$ of $F_{r-1}(x)$ mod $p$, then which point of $\overline Z_{r-1}$ must be added to this new root to get a root $z^{(r-2)}$ of $F_{r-2}(x)$ mod $p$, and so on until we get to $z^{(0)} \in \overline Z_{0}$, we can determine which point $P$ of $\P^1(\Fl)$ corresponds to $z$. We index this $z$ by a vector $v$ of $V_r$ whose reduction to $\P^1(\Fl)$ is $P$.

Now that we have indexed one root of $F_r(x)$, we index the other ones as follows: Let $\Phi$ be an arbitrary lift of $\overline \Phi \in \PGLFl$ to $\GLFl/S_r$. We know that the Frobenius of $\overline K_r$ acts as $\lambda \Phi$ for some $\lambda \in \Flx/S_r$. If we knew the value of $\lambda$, we would be able to complete the indexation of $\overline Z_{r}$ by $V_r$, since $z^{p^j}$ must be indexed by $(\lambda \Phi)^j v$ for all $j < 2^r (\ell+1)$. Each value of $\lambda$ thus corresponds to a candidate indexation of $\overline Z_{r}$ by $V_r$. In order to find out which is the correct one, we use the Dokchitsers' resolvents $\Gamma_C(x)$, albeit in an unusual way : we lift the elements of $\overline Z_{r}$ to some moderate $p$-adic precision in $\Q_p[t]/F_r(t)$, and we compute one coefficient of one of the resolvents $\Gamma_C(x)$ for each of these candidate indexations. The point is that we expect the correct indexation to yield a nice value, and the other ones to yield rubbish. Curiously, the wrong indexations yield values which are still rational over\footnote{This fact can be proved by a painful computation which we do not reproduce here.} $\Q_p$; however, in practice they will contradict archimedian bounds which can be derived from the modulus of the complex roots of $F_r(x)$, and so we can rigorously tell the right indexation apart from the wrong ones.

\begin{rk}
Let $\Gamma_C(x) = \prod_{\sigma \in C} \left( x - \sum_{z \in Z_{r}} \sigma(z) h(z) \right)$ be the resolvent whose coefficient we compute, where $h(x) \in \Z[x]$ and $C$ is a conjugacy class, and let $n=\# C$ be its degree. Clearly, the coefficients of $x^n$, of $x^{n-1}$ and of $x^0$ do not depend on the indexation and therefore give no information. Besides, in practice the height of the coefficient of $x^{n-i}$ is a roughly increasing function of $i$, so a good choice is to compute the coefficient of $x^{n-2}$, which can be done quickly by expanding the product to order $2$ at infinity.
\end{rk}

\begin{rk}
If $r$ is large, it may be better to determine the image of the Frobenius in $\GLFl/S_i$ inductively on $i=1,\cdots,r$, since this reduces the number of trials to perform from $2^r$ to $2r$. On the other hand, in practice $r$ is small (recall that $2^r < \ell$), so one may parallelise and treat all of the $2^r$ cases at once if one has enough cores to spare.
\end{rk}

\begin{rk} If we have some information about the trace or the determinant the image by $\rho$ of the Frobenius at $p$, we may make a partial prediction on which indexation is the correct one. However, we have not proved yet that the Galois-set $Z_r$ affords $\saint \rho^{S_r}$ and not another Galois representation, so to be rigorous we must try out all the possibilities.
\end{rk}

Once we know the correct indexation of $\overline Z_{r}$, we may compute the Dokchitsers' resolvents $\Gamma_C(x)$ by lifting $p$-adically the roots into $Z_{r}$. Indeed, we can deduce a bound on the necessary $p$-adic precision from archimedian bounds as above. We thus get a completely proved output. 

\subsection{Certification of the representation}

Either by the geometric approach (section \ref{Galois_geom}) or by the group cohomology one (section \ref{Galois_cohom}), we have now certified that the Galois action on the set $Z_r$ of roots of $F_r(x)$ affords a quotient Galois representation $\rho^{S_r}$, for which we are able to compute the image of the Frobenius element at $v$ for almost every prime $v \in \N$ thanks to the Dokchitsers' resolvents $\Gamma_C(x)$. We are now going to explain how to certify that this representation $\rho^{S_r}$ is equivalent to the expected representation $\saint \rho^{S_r}$.

By assumption, $\rho^{S_r}$ and $\saint \rho^{S_r}$ induce the same projective representation, so there exists a Galois character
\[ \psi \colon \Gal(\overline \Q / \Q) \longrightarrow \Flx/S_r \simeq \Z/2^r\Z \]
such that $\rho^{S_r} = \saint \rho^{S_r} \otimes \psi$. Let $(p_j)_{j \in J}$ be the primes at which $K_r$ ramifies. Since we expect $\rho^{S_r}$ to be equivalent to $\saint \rho^{S_r}$, these should be the same primes as the (known) ones at which $\saint \rho^{S_r}$ ramifies, and we assume that it is indeed the case. For each $j \in J$, let
\[ a_j = \left\{ \begin{array}{ll}
r+2 & \text{ if } p_j=2, \\
1 & \text{else},
\end{array} \right. \]
so that $\Z_{p_j}^* \otimes \Z/2^r\Z \simeq (\Z/p_j^{a_j}\Z)^* \otimes \Z/2^r\Z$ for all $j \in J$. Since $\psi$ is unramified outside the $p_j$ and assumes values in $\Z/2^r\Z$, it factors through $\Gal\big(\Q(\mu_N)/\Q\big)$, where $N = \prod_{j \in J} p_j^{a_j}$.

\pagebreak

It then suffices to find primes $v \in \N$
\begin{itemize}
\item which span $(\Z/N\Z)^* \otimes \Z/2^r\Z$,
\item for which the Dokchitser resolvents can\footnote{There are at most finitely many exceptions.} determine the trace in $\Fl/S_r$ of the image by $\rho^{S_r}$ of the Frobenius at $v$,
\item such that this trace is nonzero,
\item and which are small enough so that we can determine the trace of the image by $\saint \rho$ of the Frobenius at $v$ (for instance, if $\saint \rho=\saint \rho_{f,\l}$, we can compute the coefficients $a_v(f) \bmod \l$ using methods based on modular symbols).
\end{itemize}
If for each of these $v$ the trace is the same for $\rho^{S_r}$ and $\saint \rho^{S_r}$, this proves that $\psi$ is trivial, so that $\rho^{S_r}$ is equivalent to $\saint \rho^{S_r}$.

\begin{rk}\label{Rk_recover_L}
In particular, it then follows that the splitting field $L_r$ of $F_r(x)$ is indeed the field $\saint L_r$ cut out by $\saint \rho^{S_r}$. Besides, since the Galois representation $\saint \rho$ can be recovered from its quotient $\saint \rho^{S_r}$ and its determinant character $\det \saint \rho$, the field $\saint L$ cut out by $\saint \rho$ is the compositum of $L_r$ and of the field cut out by $\det \saint \rho$, which is by assumption a subfield of the cyclotomic field $\Q(\mu_\ell)$. Using the $\cite{gp}$ functions \verb?polsubcyclo? and \verb?polcompositum? to compute explicitly this latter field and then its compositum with $L_r$, we can thus easily compute a nice monic polynomial in $\Z[x]$ whose splitting field is $\saint L$. This is useful since, as explained in section \ref{Polred}, the polynomial $F(x) \in \Q[x]$ of degree $\ell^2-1$ computed by the algorithm described in \cite{Moi} is usually too big to be reduced directly.
\end{rk}

\newpage

\section{Application}\label{Tables}

Let $R$ be the set of couples $(f,\l)$, where $\l$ a prime ideal of degree 1 of the Hecke field\footnote{By \emph{Hecke field} of a newform, we mean the number field generated by its Fourier coefficients.} of $f$ lying above a prime number $\ell \leqslant 31$, and $f \in S_k(1)$ a newform of level $N=1$ and weight $k < \ell$, and let $R' \subsetneq R$ be the subset formed by the couples $(f,\l)$ such that the Galois representation $\saint \rho_{f,\l}$ attached to $f \bmod \l$ is not exceptional\footnote{So we exclude precisely $\Delta \bmod 23$ and $E_4 \Delta \bmod 31$.}.

For each $(f,\l)$ in $R'$, we have used the algorithm described in \cite{Moi} to compute a polynomial $F(x) \in \Q[x]$ supposedly attached to $\saint \rho_{f,\l}$. For $\ell \neq 17$, we have then reduced each of these data by the method presented in section \ref{Polred}, thus getting a collection of polynomials $F_i(x) \in \Z[x]$, $0 \leqslant i \leqslant r = \ord_2(\ell-1)$, and we have applied the group cohomology method described in sections \ref{GalProj} and \ref{Galois_cohom} to certify that these data do define the correct Galois representations. We have finally computed the Dokchitsers' resolvents corresponding to these representations, and we have used them to determine the image in $\GL_2(\F_\l)$ (up to similarity of course) of the Frobenius at $p$ by each of these representations for the 40 first primes $p \in \N$ above $10^{1000}$, so as to illustrate the fact that huge values of $p$ are not a problem for our algorithm. In particular, we have determined the value of $a_p(f) \bmod \l$ for such $p$. All of these certified data (the reduced polynomials $F_i(x)$ with their ordered roots, the Dokchitser's resolvents, and the tables of images of Frobenius elements) may be found on the author's \href{http://www2.warwick.ac.uk/fac/sci/maths/people/staff/mascot/galreps}{webpage}.

\begin{rk}
In \cite{Moi}, we noted that it took \cite{Sage} about $30$ minutes of CPU time to compute one coefficient $a_p \bmod \l$ for $p \approx 10^{1000}$ via our Galois representation data. As we reran the computations with the certified resolvents, we realised that \cite{gp} can do the same thing in less than $1$ minute. The reason for this is that \cite{Sage} takes the time to check rigorously that $p$ is prime before starting computations mod $p$, whereas \cite{gp} does not. Amusingly, this shows that it takes much more time to find a prime number $p$ of this size than to compute $a_p \bmod \l$ by the Galois representation method.

We have certified that the $40$ values of $p$ used in the tables below are indeed prime, because we are not sure what would happen if we ran our algorithm with a composite pseudoprime. As a result, the values of $a_p \bmod \l$ displayed in these tables are completely rigorous.
\end{rk}

\newpage

In order to give an idea of the size of the objects that our algorithms manipulate, we present here two cases extracted from the aforementioned tables. Instead of representing a similarity class in $\GL_2(\F_\l)$ by a matrix as we did in \cite{Moi}, we deemed it more elegant to give its \emph{minimal} polynomial in factored form over $\F_\l$. Since $\GL_2(\Fl)$ splits into similarity classes as follows, this is a faithful representation.

\[ \hspace{-1.5cm}
\begin{array}{|c|c|c|c|c|}
\hline
\text{Type of class} & \text{Representative} & \text{Minimal polynomial} & \text{\# of classes} & \text{\# of elements in class} \\
\hline
\text{Scalar} \phantom{\Bigg\vert} & \left[ \begin{array}{cc} \lambda & 0 \\ 0 & \lambda \end{array} \right] & x-\lambda & \ell-1 & 1 \\
\begin{array}{c} \text{Split} \\ \text{semisimple} \end{array} \phantom{\Bigg\vert} & \left[ \begin{array}{cc} \lambda & 0 \\ 0 & \mu \end{array} \right] & (x-\lambda)(x-\mu) & \frac{(\ell-1)(\ell-2)}2 & \ell(\ell+1) \\
\begin{array}{c} \text{Non-split} \\ \text{semisimple} \end{array} \phantom{\Bigg\vert} & \left[ \begin{array}{cc} 0 & -n \\ 1 & t \end{array} \right] & \begin{array}{c} x^2-tx+n \\ \text{irreducible over }\Fl \end{array} & \frac{\ell(\ell-1)}2 & \ell(\ell-1) \\
\text{Non-semisimple} \phantom{\Bigg\vert} & \left[ \begin{array}{cc} \lambda & 1 \\ 0 & \lambda \end{array} \right] & (x-\lambda)^2 & \ell-1 & (\ell+1)(\ell-1) \\
\hline
\end{array}
\]

\bigskip

\newpage

\noindent\textbf{Example 1:} $\Delta \bmod 29$.

\bigskip

It seems natural to start with an example with $f=\Delta=q-24q^2+252q^3+O(q^4)$, the most famous cuspform of all. While for $\ell=31$ we have $r=1$, for $\ell=29$ we have $r=2$, so the polynomials $\saint F_r(x)$ are more impressive for $\ell=29$ than for $\ell=31$. Here is the one corresponding to $\Delta \bmod 29$:

\rowcolors{2}{white}{white}
\begin{tiny}
\[ \hspace{-2cm} \arraycolsep=1.4pt\begin{array}{ccl} \saint F_2(x) & = &
x^{120} - 39 \, x^{119} + 52 \, x^{118} + 18802 \, x^{117} - 260738 \, x^{116} - 2224996 \, x^{115} + 78123651 \, x^{114} - 328828100 \, x^{113} - 8263917952 \, x^{112} \\
&& + 105418992285 \, x^{111} - 9281370047 \, x^{110} - 8673650394390 \, x^{109} + 67175813321912 \, x^{108} + 3240223696313 \, x^{107} - 3625273840703346 \, x^{106} \\
&& + 28868328866222299 \, x^{105} - 55712181926653112 \, x^{104} - 831213186859484809 \, x^{103} + 6400389530587512440 \, x^{102} + 5664948473704761298 \, x^{101} \\
&& - 236599099025809755837 \, x^{100} - 86149046526574607141 \, x^{99} + 18049361157398735512827 \, x^{98} - 143034171738473324654141 \, x^{97}  \\
&& + 309908279927036114408948 \, x^{96} + 4110452935977502930211262 \, x^{95} - 49808587507684086841613272 \, x^{94} + 255718390797761218980112249 \, x^{93} \\
&& - 370938232422515550238030706 \, x^{92} - 4239746526064029063336974560 \, x^{91} + 40059260137839079990324735682 \, x^{90} \\
&& - 205134100035408647490709294925 \, x^{89} +  690810959665321724654129463170 \, x^{88} - 1150913531696070804731460240641 \, x^{87} \\
&&  - 2905017526953691499670077418670 \, x^{86} + 47322659102097465506352390635856 \, x^{85} - 425792292478079616843046706314083 \, x^{84} \\
&& + 2739838234183913689504417826249525 \, x^{83} - 12377247662589064428784865815958075 \, x^{82} + 41296251300763242911291874924492236 \, x^{81} \\
&& - 86096254481992808573240127681847534 \, x^{80} - 174161987438617330069511957454948216 \, x^{79} + 3004945442865208465399646864785306007 \, x^{78} \\
&& - 19426609866780659578962841182962714865 \, x^{77} + 108199453121858544562274337695731535951 \, x^{76} \\
&& - 540562354485415170568171856724347249028 \, x^{75} + 2003170329279473549264139360014033008269 \, x^{74} \\
&& - 4906345350745852789161273456858421483526 \, x^{73} + 6852101959985515455407213317694533880854 \, x^{72} \\
&& + 21744835456542777978544010432017957570998 \, x^{71} - 354531601960104186814288045752985534837356 \, x^{70} \\
&& + 2415813767710375355007174048785369337370619 \, x^{69} - 11795476320637187447112847890157256430641818 \, x^{68} \\
&& + 51949786215458201865850168647651038718083533 \, x^{67} - 205837760707652251236618469331715307953868772 \, x^{66} \\
&& + 632794675891664554262532875475585224624885501 \, x^{65} - 1549984687081576409789267803107087061300626754 \, x^{64} \\
&& + 3780171680443736629265587788531817043101358021 \, x^{63} - 2032042888653854240770004273667014042737914619 \, x^{62} \\
&& - 75296586398944854033134144067268466018165634371 \, x^{61} + 492438774401604429008913700838759413140834029077 \, x^{60} \\
&& - 1872146628576921265301617989405459118651511828249 \, x^{59} + 7889534315510055163849348514205854835317146183354 \, x^{58} \\
&& - 37623219532998612719188117562544690312647851443329 \, x^{57} + 133715149099087666221878622209330023885832980173762 \, x^{56} \\
&& - 358527853259357643101016413194439711168998587653646 \, x^{55} + 1150214873720403752145704516777301458540259708566007 \, x^{54} \\
&& - 4251058748128336628769990060481020773188738825695702 \, x^{53} + 10642612653109338583300281664637819808188791020684468 \, x^{52} \\
&& - 17402914533613728148979826342208602338942607463119246 \, x^{51} + 48633429629872181118699939461795124668503022992755678 \, x^{50} \\
&& -165403276792631997282371651395087674782654230366714124 \, x^{49} + 145015997107909021398686766742679587247121061293408986 \, x^{48} \\
&& + 492392849280060573773565340461610525259317147507294865 \, x^{47} - 271511458296438382488111693610775002497465128417170394 \, x^{46}\\
&& - 652664619248620330391026643444817961046333282136405757 \, x^{45} - 16367594587199289948998686451709338569385261309703750822 \, x^{44} \\
&& + 44978511235283376299343780035953332879799842232519914312 \, x^{43} + 19646073668559858224023650929822622112934080573795228422 \, x^{42} \\
&& + 28535167429260816202303363626597519751307292203748180524 \, x^{41} - 498090822280959521158336743012213915583277009997639543769 \, x^{40} \\
&& - 940364373679220067932549479979755134636234011579427914542 \, x^{39} + 2521673052520748698612222377227238872725904760567919548740 \, x^{38} \\
&& + 7019283132304011272238795849686785307621156377148940945457 \, x^{37} + 12407898598890801572422838737227607844456571501921254925864 \, x^{36} \\
&& - 54774940542932812395031549315157134292675987516857162936933 \, x^{35} \\
&& - 167280160291743112243902528169268456978957939558833200506384 \, x^{34} \\
&& + 66685231231069675353959106828906025058508433889848745908446 \, x^{33} \\
&& + 1144200200071295796141746982232629332102662041133194625544527 \, x^{32} \\
&& + 1465380778516325802890225143289120143844003938597799565942015 \, x^{31} \\
&& - 4546042233752493082553255798793744033071375504699352571051582 \, x^{30} \\
&& - 12691048529690820177670723551290387902258432599474582511011324 \, x^{29} \\
&& + 5219645215184371778852291796118549498037264765670011997356903 \, x^{28} \\
&& + 59536146913870227752311679132874695245690076312069901091973737 \, x^{27} \\
&& + 42271202746576508837242051054585488179771161211530729060009727 \, x^{26} \\
&& - 167593661120219565661536403962471583120422676161951086004048721 \, x^{25} \\
&& - 286368937487543599711899983016552475758462484909274064469481002 \, x^{24} \\
&& + 230382055771017547055677721234005290186180568652972820922049224 \, x^{23} \\
&& + 928283302209877157721534651901436783095651772196213609374878685 \, x^{22} \\
&& + 175585932223464736559299592405845533688516285207784943808278420 \, x^{21} \\
&& - 1758850016954365463305055994507463367031764582472365647306994534 \, x^{20} \\
&& - 1465327287102397863683326389027330201118347359802335300172559328 \, x^{19} \\
&& + 1773321220836307165702143644634692168610741013365613960356877087 \, x^{18} \\
&& + 2904606733860530703041514422127534636066546248303444459223252869 \, x^{17} \\
&& - 520308669130339394544399063835249522615387011157258025834606131 \, x^{16} \\
&& - 2906947132318789204808524108533368321356173905644648961284835769 \, x^{15} \\
&& - 393534993004425879883701416875089550520476893473247289746770881 \, x^{14} \\
&& + 2113255440095432232134067491875625170919662276031515339003865608 \, x^{13} \\
&& + 343521455053064377858576614861077606598382997902674984475727361 \, x^{12} \\
&& - 1980733816420089301985076580314504281378403676364093859856750280 \, x^{11} \\
&& - 841423938599508546949037276545037161554893873562770775547347936 \, x^{10} \\
&& + 1511611164721597762311281100747394082476044535180259343320913007 \, x^{9} \\
&& + 1865894071033615040665160647561792975872738246766682774064852296 \, x^{8} \\
&& + 887398778985804089226899981553259732564931621689808536397397622 \, x^{7} \\
&& + 327959598838061445269659568556871680486016836452609211222699063 \, x^{6} \\
&& + 280807031529596339466111600718026859424625249954985059771350709 \, x^{5} \\
&& + 234434262697623313809637590557065036950844063730534986852355367 \, x^{4} \\
&& + 128418383859788691330267355023441549682203671844754849186711248 \, x^{3} \\
&& + 47862235923713816575492173460515921299171434171423149409051143 \, x^{2} \\
&& + 7941532444376844604785215172809295246343317508709928231445127 \, x \\
&& - 804139180569965777035407848426442222962300357108066928039835.
\end{array} \]\end{tiny}

\newpage

The images of the Frobenius elements are the following:

\bigskip

\rowcolors{2}{gray!25}{white}
\[ \begin{array}{|ccc|}
\hline
\rowcolor{gray!50}
p & \phantom{abcdef} \vphantom{\bigg\vert}\saint \rho_{\Delta,29}(\operatorname{Frob}_p)  \phantom{abcdef} & \tau(p) \bmod 29 \\
10^{1000}+453 & \TableMinPolyVar^2 + 8\TableMinPolyVar + 24 & 21 \\
10^{1000}+1357 & \TableMinPolyVar^2 + 21\TableMinPolyVar + 1 & 8 \\
10^{1000}+2713 & \TableMinPolyVar^2 + 18\TableMinPolyVar + 20 & 11 \\
10^{1000}+4351 & \TableMinPolyVar^2 + 3 & 0 \\
10^{1000}+5733 & (\TableMinPolyVar-20)(\TableMinPolyVar-2) & 22 \\
10^{1000}+7383 & (\TableMinPolyVar-19)(\TableMinPolyVar-10) & 0 \\
10^{1000}+10401 & (\TableMinPolyVar-7)(\TableMinPolyVar-2) & 9 \\
10^{1000}+11979 & \TableMinPolyVar^2 + 22\TableMinPolyVar + 22 & 7 \\
10^{1000}+17557 & \TableMinPolyVar^2 + 27 & 0 \\
10^{1000}+21567 & (\TableMinPolyVar-23)(\TableMinPolyVar-3) & 26 \\
10^{1000}+22273 & \TableMinPolyVar^2 + 15\TableMinPolyVar + 3 & 14 \\
10^{1000}+24493 & \TableMinPolyVar^2 + 25\TableMinPolyVar + 16 & 4 \\
10^{1000}+25947 & (\TableMinPolyVar-27)(\TableMinPolyVar-15) & 13 \\
10^{1000}+27057 & \TableMinPolyVar^2 + 22\TableMinPolyVar + 23 & 7 \\
10^{1000}+29737 & (\TableMinPolyVar-23)(\TableMinPolyVar-10) & 4 \\
10^{1000}+41599 & (\TableMinPolyVar-13)(\TableMinPolyVar-5) & 18 \\
10^{1000}+43789 & (\TableMinPolyVar-18)(\TableMinPolyVar-15) & 4 \\
10^{1000}+46227 & \TableMinPolyVar^2 + 7\TableMinPolyVar + 3 & 22 \\
10^{1000}+46339 & (\TableMinPolyVar-26)(\TableMinPolyVar-8) & 5 \\
10^{1000}+52423 & (\TableMinPolyVar-17)(\TableMinPolyVar-16) & 4 \\
10^{1000}+55831 & \TableMinPolyVar^2 + 21\TableMinPolyVar + 4 & 8 \\
10^{1000}+57867 & (\TableMinPolyVar-13)(\TableMinPolyVar-11) & 24 \\
10^{1000}+59743 & \TableMinPolyVar^2 + 24\TableMinPolyVar + 2 & 5 \\
10^{1000}+61053 & \TableMinPolyVar^2 + 18\TableMinPolyVar + 21 & 11 \\
10^{1000}+61353 & (\TableMinPolyVar-24)(\TableMinPolyVar-1) & 25 \\
10^{1000}+63729 & (\TableMinPolyVar-20)(\TableMinPolyVar-1) & 21 \\
10^{1000}+64047 & \TableMinPolyVar^2 + 14\TableMinPolyVar + 6 & 15 \\
10^{1000}+64749 & \TableMinPolyVar^2 + 14\TableMinPolyVar + 28 & 15 \\
10^{1000}+68139 & (\TableMinPolyVar-12)(\TableMinPolyVar-2) & 14 \\
10^{1000}+68367 & \TableMinPolyVar^2 + 26\TableMinPolyVar + 26 & 3 \\
10^{1000}+70897 & \TableMinPolyVar^2 + 12\TableMinPolyVar + 28 & 17 \\
10^{1000}+72237 & \TableMinPolyVar^2 + 27\TableMinPolyVar + 13 & 2 \\
10^{1000}+77611 & (\TableMinPolyVar-14)(\TableMinPolyVar-13) & 27 \\
10^{1000}+78199 & (\TableMinPolyVar-17)(\TableMinPolyVar-14) & 2 \\
10^{1000}+79237 & \TableMinPolyVar^2 + 28\TableMinPolyVar + 25 & 1 \\
10^{1000}+79767 & \TableMinPolyVar^2 + 13\TableMinPolyVar + 16 & 16 \\
10^{1000}+82767 & (\TableMinPolyVar-27)(\TableMinPolyVar-13) & 11 \\
10^{1000}+93559 & \TableMinPolyVar^2 + 13\TableMinPolyVar + 17 & 16 \\
10^{1000}+95107 & (\TableMinPolyVar-25)(\TableMinPolyVar-24) & 20 \\
10^{1000}+100003 & (\TableMinPolyVar-26)(\TableMinPolyVar-13) & 10 \\
\hline
\end{array}
\]

\newpage

\phantomsection\label{31_24}

\noindent\textbf{Example 2:} $f_{24} \bmod 31$.

\bigskip

For the second example, we pick
\[f=f_{24}=q + 24(22+\alpha) q^2 + 36(4731-32 \alpha)q^3 + O(q^4), \]
the unique (up ot Galois conjugacy) newform of level 1 and of weight $24$, because it is the one of lowest weight whose Hecke field is strictly larger than $\Q$. More precisely, the Hecke field of $f_{24}$ is the real quadratic field $\Q(\alpha)$, $\alpha = \frac{1+\sqrt{144169}}2$. Its ring of integers is $\Z[\alpha]$.

In this field, the prime $31$ splits into $(31) = \l_5 \l_{27}$, where $\l_5 = (31,\alpha-5)$ and \linebreak $\l_{27} = (31,\alpha-27)$. Instead of presenting the results for the Galois representations attached to $f_{24}$ modulo $\l_5$ and $\l_{27}$ separately, it is more interesting to present them together, since we can then compute the coefficients $\tau_{24}(p) \bmod 31 \Z[\alpha]$ by putting together the information coming from both representations and using Chinese remainders. This is what we do in the table below.

Since $\ell=31$, we have $r=1$. The polynomial $\saint F_r(x)$ corresponding to $\saint \rho_{f_{24},\l_5}$ is

\rowcolors{2}{white}{white}
\begin{tiny}
\[ \hspace{-1.5cm} \arraycolsep=1.4pt \begin{array}{rcl} \saint F_{1}(x) & = &
x^{64} - 26 \, x^{63} + 138 \, x^{62} + 2883 \, x^{61} - 50530 \, x^{60} + 284952 \, x^{59} + 1532392 \, x^{58} - 42378023 \, x^{57} + 313778342 \, x^{56} - 30967109 \, x^{55} \\
&& - 15952723659 \, x^{54} + 120293225685 \, x^{53} - 294956419293 \, x^{52} - 2450725406897 \, x^{51} + 28694976228508 \, x^{50} - 82028806284207 \, x^{49} \\
&& - 33797566443141 \, x^{48} + 30936396673955 \, x^{47} - 25385922046683633 \, x^{46} + 285017809626505879 \, x^{45} - 101340567457478942 \, x^{44} \\
&& - 5967948306452799555 \, x^{43} + 18835587705819950118 \, x^{42} - 144943245205521339710 \, x^{41} + 602219044044458739742 \, x^{40} \\
&& + 2200535330299713709469 \, x^{39} - 16686864181478594950667 \, x^{38} + 107977341642646415867192 \, x^{37} - 475668786864492416295472 \, x^{36} \\
&& - 225298037681795144992586 \, x^{35} + 13039469950621100673089867 \, x^{34} - 37880916977102172639162818 \, x^{33} \\
&& + 23877972000622578505000183 \, x^{32} - 379716355409906474595592883 \, x^{31} - 358561841745924661422683747 \, x^{30} \\
&& + 21467502653993360143238405812 \, x^{29} - 62531950374059451763223031677 \, x^{28} - 141363172107640187136259273515 \, x^{27} \\
&& + 920893472769088633347279277260 \, x^{26} - 764513501934547521440643050277 \, x^{25} - 2227564891412996848197832943852 \, x^{24} \\
&& + 471803614818821627606852431704 \, x^{23} - 6403474778189117882143498765256 \, x^{22} + 128945287900586639765937294055323 \, x^{21} \\
&& - 267130197468879823675069343083282 \, x^{20} - 609942322537763774798637252351357 \, x^{19} + 2843848149794156824379251546718928 \, x^{18} \\
&& - 1449008974308249876681217755422392 \, x^{17} - 8609964732085444739115712428740443 \, x^{16} + 11462233793731819908607681612424601 \, x^{15} \\
&& + 16721010272893391334932201233417682 \, x^{14} - 29850257116492845020236438390839168 \, x^{13} - 85528053082348511322543845120538291 \, x^{12} \\
&& + 288505635781109866818884753868632113 \, x^{11} - 35293229333983240796518647599225700 \, x^{10} - 1277262158496478519737058759156656914 \, x^{9} \\
&& + 1834010042289159626253642058051818796 \, x^{8} + 1354316757902805387817418179095807350 \, x^{7} - 4163881920776421128809003897947900249 \, x^{6} \\
&& + 988630283825310945520835533908582035 \, x^{5} + 2040826308855028479392640356469898542 \, x^{4} - 781074320529157534608502496794137429 \, x^{3} \\
&& + 709576849443416690978774803765082127 \, x^{2} - 1543465475906955668641522308642611594 \, x + 688413259803358313348163539065291572, \\
\end{array} \]\end{tiny}

and the one corresponding to $\saint \rho_{f_{24},\l_{27}}$ is

\begin{tiny}
\[ \hspace{-1.5cm} \arraycolsep=1.4pt \begin{array}{rcl} \saint F_{1}(x) & = &
x^{64} - 13 \, x^{63} - 12 \, x^{62} + 1798 \, x^{61} - 2480 \, x^{60} - 301351 \, x^{59} + 2427920 \, x^{58} + 3549779 \, x^{57} - 128622131 \, x^{56} - 605195516 \, x^{55} \\
&& + 18083445605 \, x^{54} - 76623104240 \, x^{53} - 136111338385 \, x^{52} + 163365709662 \, x^{51} + 36207027735933 \, x^{50} - 333393729013025 \, x^{49} \\
&& + 1353870749023624 \, x^{48} - 4874235588482263 \, x^{47} + 57952977575049072 \, x^{46} - 607896973953769424 \, x^{45} + 3885848486411353707 \, x^{44} \\
&& - 19706433793139872315 \, x^{43} + 120488579146025627521 \, x^{42} - 883909787742651393957 \, x^{41} + 5725316882860134327765 \, x^{40} \\
&& - 30772173337138099500438 \, x^{39} + 159943917207673058062651 \, x^{38} - 902780142644635221738911 \, x^{37} + 5191270923286965360402518 \, x^{36} \\
&& - 27218300530032866515284399 \, x^{35} + 131834043223355056977306359 \, x^{34} - 634566137578102285193778876 \, x^{33} \\
&& + 3121681910932332495500670500 \, x^{32} - 14916061491879244185623832302 \, x^{31} + 66502847707000774372555381722 \, x^{30} \\
&& - 280063144491158854648848327512 \, x^{29} + 1151797920191329188089219069705 \, x^{28} - 4647562082419563017250271030629 \, x^{27} \\
&& + 17964227685904653209413452332198 \, x^{26} - 65006898495556449638155640530135 \, x^{25} + 220529771543741523242617521771165 \, x^{24} \\
&& - 708030865546251742399340304689884 \, x^{23} + 2183095437906409520271539169052977 \, x^{22} - 6466045440189753384271760806624755 \, x^{21} \\
&& + 18519022770605982324844617113128582 \, x^{20} - 50903095666736365236595239907177352 \, x^{19} + 135712299725345417719982183578217245 \, x^{18} \\
&& - 349024414927084414313298879270239332 \, x^{17} + 879282617681138593506051646342160011 \, x^{16} - 2128887636785999977543247137539912626 \, x^{15} \\
&& + 4959567391946018954079733252123119870 \, x^{14} - 10698310092805038208309504750205888318 \, x^{13} + 21185126053660446928251211870565927064 \, x^{12} \\
&& - 37034974052822943124568751376502208132 \, x^{11} + 57682303937811470679764738932557333147 \, x^{10} - 77659172323156765855997312303575730246 \, x^{9} \\
&& + 91059874206416211006654087253008834453 \, x^{8} - 92285656456264804316815032164880452414 \, x^{7} + 79794573183910939847907389673931597531 \, x^{6} \\
&& - 60780767548452665962995019987085052653 \, x^{5} + 37996038264233396745310228794005562702 \, x^{4} - 20277402785975735994777964167007154402 \, x^{3} \\
&& + 7574966450629297705011250772005345004 \, x^{2} - 1351637429742600734951332369647381173 \, x + 193569924383211730931468549048466113.
\end{array} \]\end{tiny}

\newpage

The images of the Frobenius elements are the following:

\bigskip

\rowcolors{2}{gray!25}{white}
\[ \hspace{-1cm} \begin{array}{|cccc|}
\hline
\rowcolor{gray!50}
p & \vphantom{\bigg\vert} \saint \rho_{f_{24},\l_5}(\operatorname{Frob}_p)&  \saint \rho_{f_{24},\l_{27}}(\operatorname{Frob}_p) & a(f_{24},p) \bmod 31 \Z[\alpha] \\
10^{1000}+453 & \TableMinPolyVar^2 + 26\TableMinPolyVar + 21 & (\TableMinPolyVar-20)(\TableMinPolyVar-15) & 1+7\alpha \\
10^{1000}+1357 & (\TableMinPolyVar-18)(\TableMinPolyVar-3) & (\TableMinPolyVar-25)(\TableMinPolyVar-22) & 1+4\alpha \\
10^{1000}+2713 & (\TableMinPolyVar-24)(\TableMinPolyVar-2) & (\TableMinPolyVar-29)(\TableMinPolyVar-7) & 4+23\alpha \\
10^{1000}+4351 & (\TableMinPolyVar-17)(\TableMinPolyVar-13) & (\TableMinPolyVar-11)(\TableMinPolyVar-6) & 9+29\alpha \\
10^{1000}+5733 & (\TableMinPolyVar-19)(\TableMinPolyVar-12) & (\TableMinPolyVar-15)(\TableMinPolyVar-9) & 3+18\alpha \\
10^{1000}+7383 & \TableMinPolyVar^2 + 4\TableMinPolyVar + 14 & (\TableMinPolyVar-7)(\TableMinPolyVar-2) & 17+2\alpha \\
10^{1000}+10401 & (\TableMinPolyVar-22)(\TableMinPolyVar-5) & \TableMinPolyVar^2 + 24\TableMinPolyVar + 17 & 9+16\alpha \\
10^{1000}+11979 & \TableMinPolyVar^2 + 17\TableMinPolyVar + 7 & \TableMinPolyVar^2 + 19\TableMinPolyVar + 7 & 6+14\alpha \\
10^{1000}+17557 & (\TableMinPolyVar-26)(\TableMinPolyVar-24) & (\TableMinPolyVar-17)(\TableMinPolyVar-13) & 1+16\alpha \\
10^{1000}+21567 & \TableMinPolyVar^2 + 6\TableMinPolyVar + 29 & \TableMinPolyVar^2 + 2\TableMinPolyVar + 29 & 10+3\alpha \\
10^{1000}+22273 & \TableMinPolyVar^2 + 10\TableMinPolyVar + 19 & (\TableMinPolyVar-16)(\TableMinPolyVar-7) & 29+17\alpha \\
10^{1000}+24493 & (\TableMinPolyVar-22)(\TableMinPolyVar-12) & (\TableMinPolyVar-25)(\TableMinPolyVar-18) & 8+30\alpha \\
10^{1000}+25947 & (\TableMinPolyVar-15)(\TableMinPolyVar-12) & (\TableMinPolyVar-24)(\TableMinPolyVar-23) & 14+15\alpha \\
10^{1000}+27057 & \TableMinPolyVar^2 + 10\TableMinPolyVar + 30 & (\TableMinPolyVar-26)(\TableMinPolyVar-25) & 17+7\alpha \\
10^{1000}+29737 & \TableMinPolyVar^2 + 3\TableMinPolyVar + 24 & \TableMinPolyVar^2 + 13\TableMinPolyVar + 24 & 19+8\alpha \\
10^{1000}+41599 & \TableMinPolyVar^2 + 11\TableMinPolyVar + 8 & \TableMinPolyVar^2 + 27\TableMinPolyVar + 8 & 18+19\alpha \\
10^{1000}+43789 & \TableMinPolyVar^2 + 14\TableMinPolyVar + 3 & \TableMinPolyVar^2 + 7\TableMinPolyVar + 3 & 14+13\alpha \\
10^{1000}+46227 & \TableMinPolyVar^2 + 15\TableMinPolyVar + 12 & \TableMinPolyVar^2 + 4\TableMinPolyVar + 12 & 29+16\alpha \\
10^{1000}+46339 & (\TableMinPolyVar-24)(\TableMinPolyVar-9) & \TableMinPolyVar^2 + 5\TableMinPolyVar + 30 & 5+18\alpha \\
10^{1000}+52423 & (\TableMinPolyVar-10)(\TableMinPolyVar-1) & \TableMinPolyVar^2 + 16\TableMinPolyVar + 10 & 27+3\alpha \\
10^{1000}+55831 & \TableMinPolyVar^2 + 7\TableMinPolyVar + 25 & (\TableMinPolyVar-28)(\TableMinPolyVar-2) & 17+20\alpha \\
10^{1000}+57867 & \TableMinPolyVar^2 + 12\TableMinPolyVar + 6 & \TableMinPolyVar^2 + 6\TableMinPolyVar + 6 & 12+20\alpha \\
10^{1000}+59743 & \TableMinPolyVar^2 + 16\TableMinPolyVar + 12 & (\TableMinPolyVar-21)(\TableMinPolyVar-5) & 28+16\alpha \\
10^{1000}+61053 & (\TableMinPolyVar-18)(\TableMinPolyVar-16) & \TableMinPolyVar^2 + 15\TableMinPolyVar + 9 & 24+2\alpha \\
10^{1000}+61353 & (\TableMinPolyVar-26)(\TableMinPolyVar-13) & \TableMinPolyVar^2 + 30\TableMinPolyVar + 28 & 11+18\alpha \\
10^{1000}+63729 & \TableMinPolyVar^2 + 4\TableMinPolyVar + 23 & (\TableMinPolyVar-18)(\TableMinPolyVar-3) & 3+11\alpha \\
10^{1000}+64047 & (\TableMinPolyVar-19)(\TableMinPolyVar-3) & (\TableMinPolyVar-13)(\TableMinPolyVar-2) & 25+18\alpha \\
10^{1000}+64749 & (\TableMinPolyVar-13)(\TableMinPolyVar-10) & (\TableMinPolyVar-17)(\TableMinPolyVar-4) & 15+14\alpha \\
10^{1000}+68139 & \TableMinPolyVar^2 + 2\TableMinPolyVar + 26 & (\TableMinPolyVar-19)(\TableMinPolyVar-3) & 1+18\alpha \\
10^{1000}+68367 & (\TableMinPolyVar-22)(\TableMinPolyVar-2) & \TableMinPolyVar^2 + 21\TableMinPolyVar + 13 & 30+5\alpha \\
10^{1000}+70897 & \TableMinPolyVar^2 + 8\TableMinPolyVar + 25 & (\TableMinPolyVar-26)^2 & 15+14\alpha \\
10^{1000}+72237 & (\TableMinPolyVar-11)(\TableMinPolyVar-2) & (\TableMinPolyVar-12)(\TableMinPolyVar-7) & 6+20\alpha \\
10^{1000}+77611 & \TableMinPolyVar^2 + 5\TableMinPolyVar + 15 & \TableMinPolyVar^2 + 28\TableMinPolyVar + 15 & 27+6\alpha \\
10^{1000}+78199 & (\TableMinPolyVar-30)(\TableMinPolyVar-28) & (\TableMinPolyVar-25)(\TableMinPolyVar-15) & 17+2\alpha \\
10^{1000}+79237 & \TableMinPolyVar^2 + 10\TableMinPolyVar + 26 & (\TableMinPolyVar-27)(\TableMinPolyVar-9) & 19+19\alpha \\
10^{1000}+79767 & (\TableMinPolyVar-15)(\TableMinPolyVar-6) & (\TableMinPolyVar-7)(\TableMinPolyVar-4) & 12+8\alpha \\
10^{1000}+82767 & (\TableMinPolyVar-13)(\TableMinPolyVar-3) & (\TableMinPolyVar-24)(\TableMinPolyVar-21) & 8+14\alpha \\
10^{1000}+93559 & (\TableMinPolyVar-15)(\TableMinPolyVar-10) & \TableMinPolyVar^2 + 8\TableMinPolyVar + 26 & 17+14\alpha \\
10^{1000}+95107 & (\TableMinPolyVar-28)(\TableMinPolyVar-20) & (\TableMinPolyVar-18)(\TableMinPolyVar-7) & 18+6\alpha \\
10^{1000}+100003 & \TableMinPolyVar^2 + 21\TableMinPolyVar + 8 & (\TableMinPolyVar-10)(\TableMinPolyVar-7) & 7+13\alpha \\
\hline
\end{array} \]

\newpage

\rowcolors{2}{white}{white}

\end{document}